\documentclass[12pt,a4paper,leqno]{amsart}
\usepackage[T1]{fontenc}
\pdfoutput=1

\usepackage{amssymb,amsfonts,mathtools,mathrsfs}   
\usepackage{graphicx}
\usepackage{placeins}
\usepackage{enumerate}
\usepackage{tikz}
\usepackage{setspace}
\usepackage{xcolor}
\usepackage{hyperref}
\hypersetup{colorlinks=true, linkcolor=blue, citecolor=red}
\usepackage{enumitem}

\usepackage[
  margin=1in,
  includefoot,
  footskip=15pt,
]{geometry}

\newtheorem{thm}{Theorem}[section]
 
\newtheorem{lem}[thm]{Lemma} 
\newtheorem{prop}[thm]{Proposition}

\theoremstyle{definition} 
\newtheorem{defn}[thm]{Definition}

\theoremstyle{remark}

\numberwithin{equation}{section}

\newcommand{\R}{\mathbb{R}}

\newcommand{\N}{\mathbb{N}}

\newcommand{\cC}{\mathscr{C}}

\newcommand{\dist}{\operatorname{dist}}
\newcommand{\supp}{\operatorname{supp}}

\newcommand{\PV}{\mathrm{P.V.}}

\newcommand{\dd}{\mathop{}\!\mathrm d} 
 
\newcommand{\Per}{\operatorname{Per}}
\newcommand{\loc}{\text{loc}}

\usepackage{newpxtext}
\renewcommand\mathcal[1]{\text{\usefont{OMS}{cmsy}{m}{n}#1}}

\allowdisplaybreaks

\begin{document}

\title{A half-space theorem for nonlocal minimal surfaces}

\author[M. Cozzi]{Matteo Cozzi}
\address{(M. Cozzi) Università degli Studi di Milano, Dipartimento di Matematica ``Federigo Enriques'', via Saldini 50, 20133 Milan, Italy.}
\email{matteo.cozzi@unimi.it}

\author[J. Thompson]{Jack Thompson}
\address{(J. Thompson) The University of Western Australia (M019), 35 Stirling Highway, Perth WA 6009, Australia.}
\email{jack.thompson@uwa.edu.au}




\dedicatory{}

\begin{abstract}
We establish a half-space theorem \`a la Hoffman and Meeks for nonlocal minimal surfaces. Differently from the classical case, our result holds in every dimension.
\end{abstract}

\maketitle

\section{Introduction}

\noindent

\subsection{Minimal surfaces and the classical half-space theorem}
The study of minimal surfaces is a foundational research topic at the intersection of geometry, topology, measure theory, and partial differential equations. From the perspective of De Giorgi's theory of sets of finite perimeter, a minimal (hyper)surface is the boundary of a set that is stationary for the perimeter functional---a property that can be equivalently characterised as having zero mean curvature at sufficiently regular points. Minimal surfaces appear frequently in shape optimisation problems and have wide ranging applications in many applied disciplines, including biology, materials science, engineering, and even architecture---see, e.g.,~\cite{P17,E13,SCFQM26} for more details.

Over the last~250 years, mathematicians have investigated minimal surfaces far and wide, focusing on producing explicit examples, understanding their regularity features, and obtaining classification results. In~1990, Hoffman and Meeks~\cite{HM90} obtained a particularly elegant rigidity result, usually referred to as the \emph{half-space theorem}. In most simple terms (and in a language coherent with that of the remainder of this paper), their result can be stated as follows---see~\cite[Theorem~1]{HM90} or~\cite[Theorem~8.15]{CM11} for other more general formulations.

\begin{thm}[Hoffman and Meeks~\cite{HM90}] \label{thm:HM}
    Let~\(E\subset \R^3\) be an open set contained in a half-space and whose boundary is a smooth connected minimal surface. Then,~\(E\) is a half-space. 
\end{thm}

Theorem~\ref{thm:HM} poses significant limitations on how a (regular) minimal surface can sit in the three-dimensional space. Remarkably, this property is actually particular to the three-dimensional space. Indeed, in higher dimensions there exist smooth minimal hypersurfaces---generalised catenoids, for instance---which are trapped between two parallel hyperplanes. Our goal in this note is to establish an analogous result to Theorem~\ref{thm:HM} for \emph{nonlocal} minimal surfaces, a different class of boundaries that we now proceed to describe.

\subsection{Nonlocal minimal surfaces}
In their highly influential 2010 paper~\cite{CRS10}, Caffarelli, Roquejoffre, and Savin introduced a notion of nonlocal or fractional perimeter. Broadly speaking, fractional perimeters are a family (indexed over a parameter~$s \in (0, 1)$) of nonlocal analogues of the classical perimeter, which take into account long-range interactions between a set and its complement, instead of infinitesimal ones. Since their inception, a number of works have been devoted to the study of minimisers of fractional perimeters---sets that we shall call here \emph{minimising nonlocal minimal surfaces}. As a result of this effort, great progress has been made in understanding their regularity and rigidity properties---without pretending to offer an exhaustive list, see the papers~\cite{CRS10,SV13,BFV14,FV17,CSV19,CC19,CCS20,CDSV23} as well as the recent surveys~\cite{C22,DV23,S24}.

A minimiser~$E \subset \R^{n + 1}$ of the fractional~$s$-perimeter satisfies the equation
\begin{equation} \label{Hs=0}
\mathrm H_{s,E} = 0 \quad \mbox{on } \partial E,
\end{equation}
where, for~$x \in \partial E$,
\begin{align*}
\mathrm H_{s, E}(x) \coloneqq & \,\, \PV \int_{\R^{n + 1}} \frac{\chi_{\R^{n + 1} \setminus E}(y) - \chi_E(y)}{|x - y|^{n + 1 + s}} \dd y \\
= & \, \lim_{\varepsilon \rightarrow 0^+} \int_{\R^{n + 1} \setminus B_\varepsilon(x)} \frac{\chi_{\R^{n + 1} \setminus E}(y) - \chi_E(y)}{|x - y|^{n + 1 + s}} \dd y.
\end{align*}
In analogy with the theory of classical minimal surfaces, the latter quantity is often called \emph{fractional~$s$-mean curvature}. The Euler-Lagrange equation~\eqref{Hs=0} has to be understood in a suitable viscosity sense, which becomes an actual pointwise identity at points around which~$\partial E$ is sufficiently regular. Sets~$E$ that satisfy~\eqref{Hs=0}---whose boundaries we shall simply address to as \emph{nonlocal~$s$-minimal surfaces}---are not necessarily minimisers of the~$s$-perimeter. A few examples of non-minimising nonlocal minimal surfaces have been constructed in the literature, mostly in~\cite{CDD16}, by Cinti, D\'avila, and del Pino, and~\cite{DDW18}, by D\'avila, del Pino, and Wei. In the very recent~\cite{CFS23}, Caselli, Florit-Simon, and Serra extended the study of nonlocal minimal surfaces to general closed Riemannian manifolds and established the existence of infinitely many of them, thus proving a fractional version of Yau's conjecture. In comparison with the results available for minimisers, however, much less is known at the moment about them.

\subsection{A half-space theorem for nonlocal minimal surfaces}
In this paper, we provide a novel rigidity result for nonlocal minimal surfaces, which acts as a fractional counterpart of the classical half-space theorem---i.e., Theorem~\ref{thm:HM} above. In its simplest form, it can be stated is as follows.

\begin{thm} \label{thm:s-halfspace}
Let~$n \in \N$ and~$s \in (0, 1)$. Let~\(E\subset \R^{n+1}\) be an open set contained in a half-space and whose boundary is a smooth nonlocal~$s$-minimal surface.
Then,~\(E\) is a half-space. 
\end{thm}

Theorem~\ref{thm:s-halfspace} is among the first rigidity results for non-minimising~$s$-minimal surfaces. In~\cite{CV20}, Valdinoci and the first author obtained what is perhaps the closest result available in the literature, by showing that hyperplanes are the only nonlocal minimal surfaces that can be trapped inside a slab---or, more generally, between the graphs of two functions that are strictly sublinear at infinity. Notice that Theorem~\ref{thm:s-halfspace} here improves upon the simplest, ``rigidity-inside-a-slab'' formulation of~\cite[Theorem~1]{CV20}, by making its assumption one-sided. Other results in this direction are~\cite[Theorem~1.6]{CC19}, by Cabr\'e and the first author, which established the flatness of nonlocal minimal graphs growing at most linearly at infinity, and~\cite[Theorem~1.3]{CFL21}, by the first author, Farina, and Lombardini, that weakened said growth assumption by requiring it only from either above or below. Related results are also~\cite[Theorem~1.4]{FV19}, by Farina and Valdinoci, and~\cite[Theorem~8.1]{CC19}, which provided the flatness of globally Lipschitz nonlocal minimal graphs. We point out that, apart from~\cite{CV20}, all these results pertain to fractional minimal \emph{graphs}, which, under reasonable regularity assumptions, are known to be minimisers of the fractional perimeter---see~\cite[Corollary~2.5]{C20} or~\cite[Theorem~1.10]{CL21}. Note that Theorem~\ref{thm:s-halfspace} is already known for \emph{minimising} nonlocal minimal surfaces and can be proved by a relatively simple blow-down argument---see~\cite[Lemma~8.3]{DSV20} or~\cite[Theorem~1.6]{CFL21}.

Differently from the local Theorem~\ref{thm:HM}, our Theorem~\ref{thm:s-halfspace} does not require the boundary of the set~$E$ to be connected. This is a common feature of nonlocal problems, motivated by the fact that, unlike its classical counterpart, the nonlocal mean curvature \emph{sees} disconnected components of~$\partial E$---the simplest repercussion of this feature is perhaps represented by the fact that the boundary of a slab is a classical minimal surface, but not a fractional one. More interestingly, Theorem~\ref{thm:s-halfspace} is valid in every dimension. This is in sharp contrast with Theorem~\ref{thm:HM}, which does not hold when the dimension of the ambient space is larger than three. The general rationale behind this difference should also be ascribed to the nonlocality of the fractional mean curvature, but the specific technical reasons are better understood by looking directly at the proof of Theorem~\ref{thm:s-halfspace}, which we now outline.

\subsection{Strategy of the proof}

The general strategy that we employ to prove Theorem~\ref{thm:s-halfspace} is a touching argument via a suitable barrier, similar, at its core, to the one devised by Hoffman and Meeks for Theorem~\ref{thm:HM}.

The proof of Theorem~\ref{thm:HM} essentially goes as follows. Suppose that~$E$ is contained in the lower half-space~$\R^3_- = \R^2 \times (-\infty, 0)$ and that no other half-space of the form~$\R^2 \times (-\infty, a)$, with~$a < 0$, contains~$E$. By the strong comparison principle for minimal surfaces,~$\partial E$ cannot touch the boundary of~$\R^3_-$. Hence, if~$E_\delta = E + (0, 0, \delta)$ is an upward translation of~$E$ by a vector of fixed small length~$\delta > 0$ then~\(E_\delta\) still lies inside~$\R^3_-$ in a neighbourhood of the origin, and lies below the upper half~$C_1$ of a catenoid globally in~$\R^3$. This creates the necessary room to compare~$\partial E_\delta$ with the rescalings~$C_\varepsilon = \varepsilon C_1$ of~$C_1$, for~$\varepsilon \in (0, 1]$. One then denotes as~$\varepsilon_0 \in [0, 1]$ the infimum of the values~$\varepsilon$ for which~$\partial E_\delta$ lies below~$C_\varepsilon$, verifies that, if positive, such a~$\varepsilon_0$ is attained as a minimum, and shows that it cannot be positive since, otherwise,~$\partial E_\delta$ would be touched on one side by~$\partial C_{\varepsilon_0}$, in contradiction once again with the strong comparison principle. Hence, it necessarily holds that~$\varepsilon_0 = 0$, which means that~$E_\delta$ lies below every rescaling of the half-catenoid~$C_1$---i.e., that~$E_\delta \subset \R^3_-$. Going back to the original set~$E$, this gives that~$E \subset \R^2 \times (-\infty, - \delta)$, in contradiction with the optimality of the inclusion~$E \subset \R^3_-$.

The reason why this beautiful simple proof works can be found in the features of the half-catenoid~$C_1$. Of course, it is a (super)solution to the minimal surface equation---a condition needed to apply the strong comparison principle and, ultimately, essential for the general feasibility of a touching argument. Then,~$C_1$ is a graph of a function~$\varphi$ defined outside of a neighbourhood of the origin in the base plane~$\R^2$ and diverging at infinity---a property that gives the necessary compactness to warrant that the infimum~$\varepsilon_0$ is attained. Finally, the fact that~$\varphi$ diverges logarithmically (thus, in a strictly sublinear fashion) at infinity is crucial to have that the blow-down of~$C_1$ is a plane and thus to conclude that~$E_\delta \subset \R^3_-$ if~$\varepsilon_0 = 0$.

When one tries to apply this argument to nonlocal minimal surfaces, several difficulties arise. First, an analogue of the catenoid for the~$s$-mean curvature has been constructed in~\cite{DDW18}, but only for values of~$s$ close to~$1$. Furthermore, it does not have all the required properties---indeed it does diverge at infinity, but exactly linearly. Most importantly, when one cuts the~$s$-catenoid in half and only considers its upper fold, the newly obtained surface is no longer a (super)solution of~\eqref{Hs=0}---this is, once again, a reflection of the nonlocality of the~$s$-mean curvature.

To extend the proof of Theorem~\ref{thm:HM} to nonlocal minimal surfaces, one has then to obtain a supersolution in a different way. The underlying reason for which the~$n$-dimensional catenoid works when~$n = 2$ but fails for every~$n \ge 3$ (and, somewhat poetically, actually provides a counterexample to the theorem in this regime!) is linked to \emph{parabolicity}. This notion has been introduced in the context of general Riemannian manifolds and has several equivalent characterisations, including the recurrence of the Brownian motion, the non-existence of non-constant positive superharmonic functions, and the positiveness of the capacity of compact sets---see~\cite{G99} for an extensive account. In the specific setting of the Euclidean space, parabolicity can be detected by the behaviour at infinity of the fundamental solution of the Laplacian. Since the linearisation of the mean curvature operator involves the Laplacian, minimal surfaces are in some sense akin to the graphs of harmonic functions. Indeed, the~$n$-catenoid behaves at infinity like the fundamental solution of the Laplacian---in particular, they both diverge logarithmically when~$n = 2$ (the parabolic regime for the Laplacian) and stay bounded when~$n \ge 3$ (the non-parabolic one). See also~\cite{RSS13,CMMR22,CM24} for broader formulations of the half-space theorem in the contexts of Riemannian manifolds and metric measure spaces, where parabolicity enters the proof in a different way.

Moving to the nonlocal setting, the linearisation of the~$s$-mean curvature is driven by the operator~$(-\Delta)^{\frac{1 + s}{2}}$---i.e., the fractional Laplacian of order~$\sigma \coloneqq 1 + s \in (1, 2)$---over the space~$\R^n$. The fundamental solution of the~$\frac{\sigma}{2}$-Laplacian is (a positive multiple of) the function~$\mbox{sgn}(n - \sigma) |x'|^{\sigma - n}$ when~$n \ne \sigma$ and~${- \log |x'|}$ when~$n = \sigma = 1$ (here and in the rest of the article,~$x'$ denotes a point in the base space~$\R^n$). Accordingly, the parabolic regime for these operators is achieved when~$n = 1$ and $\sigma \ge 1$. As the latter condition is always satisfied in our geometric setting, when~$n = 1$ we can obtain a set which is the subgraph of positive function diverging sublinearly at infinity---the power~$|x'|^\alpha$, for~$\alpha \in (0, s)$---and that has positive nonlocal mean curvature outside of a ball of fixed large radius---see Proposition~\ref{prop:barrier}\ref{barriercasei}. Through this comparison set, we can replicate the touching argument of Hoffmann and Meeks almost verbatim and establish Theorem~\ref{thm:s-halfspace} in a rather straightforward way. In light of the bare-bones nature of this approach, Theorem~\ref{thm:s-halfspace} for~$n = 1$ actually holds under virtually no regularity assumptions on~$\partial E$ and for mere viscosity (sub)solutions of the nonlocal mean curvature equation---see Section~\ref{sec:1dproof} and, in particular, Theorem~\ref{thm:techthm1} for a precise statement and its proof.

The non-parabolic case~$n \ge 2$ is obviously more challenging and requires genuinely nonlocal ideas which are not available for classical minimal surfaces. As a starting point, we still consider the subgraph~$F$ of the power function~$|x'|^\alpha$. This is no longer a supersolution, but its fractional~$s$-mean curvature goes to zero relatively quickly as~$|x'|$ diverges, at least as~$|x'|^{\alpha - 1 - s}$---see Proposition~\ref{prop:barrier}\ref{barriercaseii}. This fast decay rate allows us to correct~$F$ by removing from it a suitable region~$D$ and thus turn it into an actual supersolution~$\widetilde{F} \coloneqq F \setminus D$. Of course, the set~$D$ has to lie outside of the~$s$-minimal set~$E$, away from its touching point with~$\partial F$, and must determine a sufficiently large alteration to the fractional mean curvature in order to make this positive. To find this region~$D$, we devise a strategy that can be seen as a quantitative and more technically involved version of the dichotomy between ``graph-type'' and ``catenoid-type'' nonlocal minimal surfaces identified in the proof of~\cite[Theorem~1]{CV20}---see the forthcoming Section~\ref{sec:ndproof} and, in particular, Steps~3-5 there for the implementation of this idea. A key tool used for this argument is represented by a couple of novel density estimates---for~\emph{stable} nonlocal minimal surfaces and for mere stationary ones at points which satisfy a touching ball condition---that we collect in Subsection~\ref{subsec:density}. As for the case~$n = 1$, the regularity assumptions in Theorem~\ref{thm:s-halfspace} for~$n \ge 2$ can be considerably weakened, albeit not completely removed---see Theorem~\ref{thm:techthm2} (and the comment preceding it) for a sharper but more technical statement.

\subsection*{Acknowledgements}
The first author has been supported by the GNAMPA-INdAM project ``Equazioni nonlocali e nonlineari: alcune questioni di esistenza, rigidit\`a e regolarit\`a'' CUP E53C25002010001 (Italy), by the PRIN project 20229M52AS\_004 ``Partial differential equations and related geometric-functional inequalities'' (Italy), and by the grants PID2021-123903NB-I00 and RED2022-134784-T funded by ERDF ``A way of making Europe'' (European Union) and by MCIN/AEI/10.13039/501100011033 (Spain).

\section{A few auxiliary results}

\noindent
We collect here several preliminary results, including definitions of viscosity solutions for the nonlocal minimal surface equation, a strong comparison principle, the computation of a suitable barrier, and a couple of density estimates. Accordingly, this section is divided into four subsections, each pertaining to one such matter.

\subsection{Preliminaries on viscosity solutions}

In this subsection we introduce our definition of viscosity solution and show its equivalence with the one usually adopted in the framework of nonlocal minimal surfaces. From now on and unless otherwise specified, inclusions between sets are understood to hold up to sets of measure zero, i.e.,~$F \subset E$ means that~$|{E \setminus F}| = 0$. In particular, two sets~$E$ and~$F$ are equal if their symmetric difference~$E \Delta F$ has measure zero.

We begin by stating the definition of viscosity solution that we shall adopt most of the times in this paper.

\begin{defn} \label{def:viscsol}
Let~$E \subset \R^{n + 1}$ be a measurable set and~$x \in \partial E$. We say that:
\begin{itemize}[leftmargin=15pt]
\item $\mathrm H_{s, E} \ge 0$ at~$x$ in the viscosity sense if~$|{B_r(x) \setminus E}| > 0$ for all~$r > 0$ and~$\mathrm H_{s, F}(x) \ge 0$ for every set~$F \subset E$ such that~$x \in \partial F$ and~$\partial F$ is of class~$C^2$ in a neighbourhood of~$x$.
\item $\mathrm H_{s, E} \le 0$ at~$x$ in the viscosity sense if~$|{B_r(x) \cap E}| > 0$ for all~$r > 0$ and~$\mathrm H_{s, F}(x) \le 0$ for every set~$F \supset E$ such that~$x \in \partial F$ and~$\partial F$ is of class~$C^2$ in a neighbourhood of~$x$.
\item $\mathrm H_{s, E} = 0$ at~$x$ in the viscosity sense if both~$\mathrm H_{s, E} \ge 0$ and~$\mathrm H_{s, E} \le 0$ at~$x$ in the viscosity sense.
\end{itemize}
\end{defn}

The above definition of viscosity solutions of~$\mathrm H_{s, E} = 0$ is rather classical in spirit, as it offloads the verification of the inequalities involving the~$s$-mean curvature onto smooth hypersurfaces that touch~$\partial E$ from the appropriate side. It has been previously adopted by the first author in~\cite{CV20}, as well as in~\cite{CL21} in the context of nonlocal minimal graphs. Moreover, it is easy to see that, when~$\partial E$ is of class~$C^2$ in a neighbourhood of~$x$, then this notion is equivalent to the pointwise one.

In the literature about nonlocal minimal surfaces, the following definition of viscosity solution has been predominantly embraced.

\begin{defn} \label{def:viscsol2}
Let~$E \subset \R^{n + 1}$ be a measurable set and~$x \in \partial E$. We say that:
\begin{itemize}[leftmargin=15pt]
\item $E$ is a viscosity supersolution of~$\mathrm H_{s, E} = 0$ at~$x$ if~$|{B_r(x) \setminus E}| > 0$ for all~$r > 0$ and~$\mathrm H_{s, E}(x) \ge 0$ whenever~$E$ is touched by a ball from the interior at~$x$.
\item $E$ is a viscosity subsolution of~$\mathrm H_{s, E} = 0$ at~$x$ if~$|{B_r(x) \cap E}| > 0$ for all~$r > 0$ and~$\mathrm H_{s, E}(x) \le 0$ whenever~$E$ is touched by a ball from the exterior at~$x$.
\item $E$ is a viscosity solution of~$\mathrm H_{s, E} = 0$ at~$x$ if it is both a viscosity super- and subsolution.
\end{itemize}
\end{defn}

Among the many references, this definition has been used (implicitly) in~\cite[Theorem~5.1]{CRS10},~\cite[Theorem~4.1]{C20} and (explicitly, in a more general framework) in~\cite[Definition~1.8]{M25}. It is even simpler (in comparison to the previous Definition~\ref{def:viscsol}) to check that it agrees with the pointwise notion for~$\mathrm H_{s, E}(x) = 0$ when the hypersurface~$\partial E$ is~$C^2$ around~$x$. As a minor drawback, to even make sense of the definition the reader is required to know a priori that~$\mathrm H_{s, E}(x)$ is well-defined in the principal value sense (albeit possibly infinite) whenever~$E$ is touched by a ball from either side at~$x$---see~\cite[Section~4]{C20} for a proof of this fact.

We conclude this brief subsection by showing that Definitions~\ref{def:viscsol} and~\ref{def:viscsol2} are actually equivalent. In order to see this, it is clearly enough to prove the following statement.

\begin{prop} \label{prop:visc-equiv}
Let~$E \subset \R^{n + 1}$ be a measurable set and~$x \in \partial E$. Assume that~$|{B_r(x) \setminus E}| > 0$ for all~$r > 0$. Then,~$\mathrm H_{s, E} \ge 0$ at~$x$ in the viscosity sense of Definition~\ref{def:viscsol} if and only if~$E$ is a viscosity supersolution of~$\mathrm H_{s, E} = 0$ at~$x$ in the sense of Definition~\ref{def:viscsol2}.
\end{prop}
\begin{proof}
One implication is immediate. Indeed, suppose~$E$ to be a viscosity supersolution of~$\mathrm H_{s, E} = 0$ at~$x$ as per Definition~\ref{def:viscsol2} and let~$F \subset E$ be a set for which~$x \in \partial F$ and such that~$\partial F$ is of class~$C^2$ in a neighbourhood of~$x$. By this last regularity requirement, it is clear that~$F$ satisfies an interior ball condition at~$x$, i.e., that there exists a ball~$B \subset F$ with~$x \in \partial B$. As~$F$ is included in~$E$, we have that~$B \subset E$. Hence,~$\mathrm H_{s, E}(x) \ge 0$, according to Definition~\ref{def:viscsol2}. Using again that~$F \subset E$, this implies
\begin{align*}
0 & \le \mathrm H_{s, E}(x) = \PV \int_{\R^{n + 1}} \frac{\chi_{\R^{n + 1} \setminus E}(y) - \chi_E(y)}{|x - y|^{n + 1 + s}} \dd y \\
& \le \PV \int_{\R^{n + 1}} \frac{\chi_{\R^{n + 1} \setminus F}(y) - \chi_F(y)}{|x - y|^{n + 1 + s}} \dd y = \mathrm H_{s, F}(x),
\end{align*}
as required by Definition~\ref{def:viscsol}.

Conversely, suppose that~$\mathrm H_{s, E} \ge 0$ at~$x$ in the viscosity sense of Definition~\ref{def:viscsol} and that there exists a ball~$B_r(z) \subset E$ for some~$z \in \R^{n + 1}$ and~$r > 0$, with~$x \in \partial B_r(z)$. Up to a translation, rotation, and dilation, we may assume that~$x = 0$,~$z = e_1$, and~$r = 1$. Consider, for~$\varepsilon \in (0, 1)$, the set~$F_\varepsilon \coloneqq \big( {E \setminus B_\varepsilon} \big) \cup \big( {B_1(e_1) \cap B_{\varepsilon}} \big)$. Naturally,~$F_\varepsilon \subset E$,~$0 \in \partial F_\varepsilon$, and the boundary of~$F_\varepsilon$ is smooth in a neighbourhood of the origin. In view of Definition~\ref{def:viscsol}, we then have that
\begin{align*}
0 & \le \mathrm H_{s, F_\varepsilon}(0) = \PV \int_{\R^{n + 1}} \frac{\chi_{\R^{n + 1} \setminus F_\varepsilon}(y) - \chi_{F_\varepsilon}(y)}{|y|^{n + 1 + s}} \dd y \\
& = \PV \int_{B_\varepsilon} \frac{\chi_{\R^{n + 1} \setminus B_1(e_1)}(y) - \chi_{B_1(e_1)}(y)}{|y|^{n + 1 + s}} \dd y + \int_{\R^{n + 1} \setminus B_\varepsilon} \frac{\chi_{\R^{n + 1} \setminus E}(y) - \chi_{E}(y)}{|y|^{n + 1 + s}} \dd y,
\end{align*}
for every~$\varepsilon \in (0, 1)$. We now let~$\varepsilon \rightarrow 0^+$. It is easy to see that the first summand on the second line of the above inequality is of order~$\varepsilon^{1 - s}$ and thus goes to zero as~$\varepsilon \rightarrow 0^+$. Consequently,
$$
\mathrm H_{s, E}(0) = \lim_{\varepsilon \rightarrow 0^+} \int_{\R^{n + 1} \setminus B_\varepsilon} \frac{\chi_{\R^{n + 1} \setminus E}(y) - \chi_{E}(y)}{|y|^{n + 1 + s}} \dd y \ge 0,
$$
as desired.
\end{proof}

Notice that neither Definition~\ref{def:viscsol} nor Definition~\ref{def:viscsol2} \emph{see} sets of measure zero, meaning that if~$E$ is a (super-/sub-)solution of~$\mathrm H_{s, E}(x) = 0$ at some point~$x \in \partial E$ (in either of the senses prescribed by the two definitions), then so is any~$\widetilde{E} \subset \R^{n + 1}$ for which~$|{\widetilde{E} \Delta E}| = 0$. Hence, the notion of viscosity solution extends to the~$L^1_\loc(\R^{n + 1})$ class of a subset~$E$ of~$\R^{n + 1}$---or, more appropriately, of its characteristic function~$\chi_E$.

Moving forward, we shall always make a specific choice within the~$L^1_\loc(\R^{n + 1})$ class of a given set~$E \subset \R^{n + 1}$ and suppose that~$E$ contains its measure theoretic interior, does not intersect its measure theoretic exterior, and that its boundary coincides with its measure theoretic one. More explicitly, defining
\begin{align*}
E_{\rm int} & \coloneqq \Big\{ {x \in \R^{n + 1} \text{ s.t. } \big| {E \cap B_r(x)} \big| = |B_r| \mbox{ for some } r > 0} \Big\}, \\
E_{\rm ext} & \coloneqq \Big\{ {x \in \R^{n + 1} \text{ s.t. } \big| {E \cap B_r(x)} \big| = 0 \mbox{ for some } r > 0} \Big\}, \\
\partial^- E & \coloneqq \Big\{ {x \in \R^{n + 1} \text{ s.t. } 0 < \big| {E \cap B_r(x)} \big| < |B_r| \mbox{ for all } r > 0} \Big\} = \R^{n + 1} \setminus \big( {E_{\rm int} \cup E_{\rm ext}} \big),
\end{align*}
we suppose that
$$
E_{\rm int} \subset E, \quad E_{\rm ext} \cap E = \varnothing, \quad \mbox{and} \quad \partial E = \partial^- E,
$$
where inclusions and identities are understood here in the usual set theoretic sense. Note that any measurable set~$E \subset \R^{n + 1}$ can be redefined over a set of measure zero in order to fulfil such requirements---see, e.g., step two of the proof of~\cite[Proposition~12.19]{M12} or~\cite[Remark~1.9 and Appendix~C]{L19}. We also point out that, as a consequence of this choice, the requirements made in Definitions~\ref{def:viscsol} and~\ref{def:viscsol2} on the non-degeneracy of the densities of~$E$ and of its complement near a point~$x \in \partial E$ are redundant---as~$\partial E = \partial^- E$.

\subsection{A strong comparison principle}
This very brief subsection contains a strong comparison principle between a subsolution and a supersolution of the fractional mean curvature equation. As is often the case for nonlocal operators, these surfaces are required to solve their respective inequality only at the touching point---around which we require one of them to be smooth. See instead~\cite{DSV25} for a comparison principle which holds under no regularity assumptions, albeit for \emph{minimising} nonlocal minimal surfaces.

\begin{prop} \label{prop:compprinc}
    Let~$s \in (0, 1)$,~$E \subset F \subset \R^{n + 1}$ be two measurable sets, and~$x_0 \in \partial E \cap \partial F$. Suppose that~$\mathrm H_{s, E} \le 0$ and~$\mathrm H_{s, F} \ge 0$ at~$x_0$ in the viscosity sense and that either~$\partial E$ or~$\partial F$ is of class~$C^2$ in a neighbourhood of~$x_0$. Then,~$E = F$.
\end{prop}
\begin{proof}
    First of all, we observe that, without loss of generality, we may assume that~$\partial F$ is of class~$C^2$ in a neighbourhood of~$x_0$. Indeed, the case of smooth~$\partial E$ would follow from this by simply swapping~$E$ for~$\R^n \setminus F$ and~$F$ for~$\R^n \setminus E$.

    Next, we argue by contradiction and suppose that~$E \subsetneq F$. Then, there exist~$\delta > 0$ and a set~$D \subset (F \setminus E) \setminus B_\delta(x_0)$ such that
    \begin{equation} \label{Dhaspositivemeasure}
        |D| > 0.
    \end{equation}
    Let then~$F_0 \coloneqq F \setminus D$. Clearly,~$F_0$ has~$C^2$-boundary touching~$E$ from the outside at~$x_0$. Hence, we may exploit the fact that~$\mathrm H_{s, E} \le 0$ at~$x_0$ in the viscosity sense and deduce that
    $$
        \mathrm H_{s, F_0}(x_0) \le 0.
    $$
    On the other hand, as~$F$ has boundary of class~$C^2$ near~$x_0$ and satisfies~$H_{s, F} \ge 0$ at~$x_0$ in the viscosity sense, we also have
    $$
        \mathrm H_{s, F}(x_0) \ge 0.
    $$
    The combination of these two inequalities yields that
    $$
    	0 \ge \mathrm H_{s, F_0}(x_0) - \mathrm H_{s, F}(x_0) = 2 \int_{D} \frac{\dd y}{|x_0 - y|^{n + 1 + s}},
    $$
    in contradiction with~\eqref{Dhaspositivemeasure}. We then conclude that~$E = F$ as claimed.
\end{proof}

\subsection{A useful barrier}

As already discussed in the introduction, to establish Theorem~\ref{thm:s-halfspace} we will follow the classical argument of~\cite{HM90}. To this aim, we need a suitable supersolution of the~$s$-mean curvature equation, replacing the (half-)catenoid used in~\cite{HM90}. The following result provides such a barrier when~$n = 1$ and an approximate one when~$n \ge 2$.

\begin{prop} \label{prop:barrier}
    Let~$n \in \N$,~\(s\in (0,1)\),~\(\alpha \in (0,1)\), and set
    \begin{equation} \label{Fdef}
        F_{n, \alpha} \coloneqq \Big\{ { (x', x_{n + 1}) \in \R^{n + 1} \text{ s.t. } x_{n + 1}< |x'|^\alpha} \Big\}.
    \end{equation}
	\begin{enumerate}[leftmargin=23pt,label=$(\roman*)$]
		\item \label{barriercasei} If~$n = 1$ and~$\alpha \in (0, s)$, then there exists a radius~$R \ge 1$, depending only on~$s$ and~$\alpha$, such that
		$$
			\mathrm H_{s, F_{1, \alpha}}(x) > 0 \quad \mbox{for all } x \in \partial F_{1, \alpha} \setminus B_R.
		$$
		\item \label{barriercaseii} If~$n \ge 2$, then there exists a constant~$C \ge 1$, depending only on~$n$,~$s$, and~$\alpha$, such that
		$$
			\left| \mathrm H_{s, F_{n, \alpha}}(x) \right| \le C |x|^{\alpha - 1 - s} \quad \mbox{for all } x \in \partial F_{n, \alpha} \setminus B_1.
		$$
	\end{enumerate}
\end{prop}

To prove Proposition~\ref{prop:barrier}, we will use the representation of the~$s$-mean curvature of a subgraph in terms of an integral operator acting on its defining function. Indeed, if
$$
E = \Big\{ {(x', x_{n + 1}) \in \R^{n + 1} \text{ s.t. } x_{n + 1} < u(x')} \Big\},
$$
for some measurable function~$u: \R^n \to \R$, then
\begin{equation} \label{s-curvofgraphs}
\mathrm H_{s, E}(x) = 2 \, \PV \int_{\R^n} G_s \bigg( {\frac{u(x') - u(y')}{|x' - y'|}} \bigg) \, \frac{\dd y'}{|x' - y'|^{n + s}},
\end{equation}
with~$G_s: \R \to \R$ defined by
\begin{equation} \label{Gdef}
G_s(t) \coloneqq \int_0^t \frac{\dd \tau}{\big( {1 + \tau^2} \big)^{\frac{n + 1 + s}{2}}} \quad \mbox{for } t \in \R,
\end{equation}
at every point~$x = (x', u(x'))$ around which~$\partial E$ is of class~$C^2$---see, e.g.,~\cite{AV14,BFV14,BLV19}. Through this representation, it is apparent that the~$s$-mean curvature of the subgraph~$E$ is well-approximated by the~$\frac{1 + s}{2}$-Laplacian of~$u$ near points with small gradients. To make this fact rigorous, we will take advantage of the following simple estimate.

\begin{lem} \label{Gslem}
    Let~$s \in (0, 1)$ and~$\widetilde{G}_s: \R \to \R$ be the odd function defined by
    $$
    \widetilde{G}_s(t) \coloneqq G_s(t) - t \quad \mbox{for } t \in \R,
    $$
	with~$G_s$ given by~\eqref{Gdef}. Then,
    \begin{equation} \label{Gsest}
    \big| {\widetilde{G}_s(a) - \widetilde{G}_s(b)} \big| \le \frac{n + 2}{2} \, \max \big\{ {a^2, b^2} \big\} |a - b| \quad \mbox{for every } a, b \in \R.     
    \end{equation}
\end{lem}
\begin{proof}
    We compute
    \begin{align*}
        \big| {\widetilde{G}_s(a) - \widetilde{G}_s(b)} \big| & = \big| {G_s(a) - G_s(b) - a + b} \big| = \bigg| {\int_b^a \bigg( {\frac{1}{\big({1 + \tau^2}\big)^{\frac{n + 1 + s}{2}}} - 1} \bigg) \dd \tau} \bigg| \\
        & = (n + 1 + s) \, \bigg| {\int_b^a \bigg( {\int_0^\tau \frac{\sigma}{\big({1 + \sigma^2}\big)^{\frac{n + 3 + s}{2}}} \dd \sigma} \bigg) \dd \tau} \bigg|.
    \end{align*}
    Notice that the map
    $$
    \R \ni \tau \longmapsto \int_0^\tau \frac{\sigma}{\big({1 + \sigma^2}\big)^{\frac{n + 3 + s}{2}}} \dd \sigma
    $$
    is even and non-negative. Assuming, without loss of generality, that~$b < a$, we then obtain
    \begin{align*}
       \big| {\widetilde{G}_s(a) - \widetilde{G}_s(b)} \big| & = (n + 1 + s) \int_b^a \bigg( {\int_0^{|\tau|} \frac{\sigma}{\big({1 + \sigma^2}\big)^{\frac{n + 3 + s}{2}}} \dd \sigma} \bigg) \dd \tau \\
       & \le (n + 1 + s) \int_b^a \bigg( {\int_0^{|\tau|} \sigma \dd \sigma} \bigg) \dd \tau \le \frac{n + 2}{2} \max \big\{ {a^2, b^2} \big\} (a - b),
    \end{align*}
    as desired.
\end{proof}

To address the case~$n = 1$ in Proposition~\ref{prop:barrier} the following result will be of key importance. Even though it is already known in the literature---it can be deduced, for instance, from the more general~\cite[Proposition~2.10]{ADV25}---, we provide a proof in full details for the reader's convenience.

\begin{lem} \label{GO1TWVl8}
    Let \(s\in (0,1)\) and \(\alpha \in (0, 1+s)\). Then
	$$
        \PV \int_{\R}\frac{1- \vert t\vert^\alpha}{\vert 1-t\vert^{2+s} } \dd t = \lim_{\varepsilon \to 0^+} \int_0^{1-\varepsilon}  \big ( 1- t^\alpha\big )\big (1 -t^{s-\alpha} \big )\bigg [ \frac1{(1-t )^{2+s} } + \frac1{( t+1 )^{2+s} } \bigg ] \dd t .
    $$
	In particular,
	$$
        \PV \int_{\R}\frac{1- \vert t\vert^\alpha}{\vert 1-t\vert^{2+s} } \dd t \, \begin{cases}
              \, >0 & \quad \text{if } \alpha \in (0, s), \\
              \, =0 & \quad \text{if } \alpha=s, \\
              \, <0 & \quad \text{if } \alpha \in (s, 1 + s).
        \end{cases}
    $$
\end{lem}
\begin{proof}
    In this proof,~$C$ denotes a positive constant depending only on~$s$ and~$\alpha$, whose value may change from line to line.

	Fix~\(\varepsilon \in (0, 1)\) small. Then,
	\begin{align*}
        \int_{\R\setminus (1-\varepsilon,1+\varepsilon)}\frac{1- \vert t\vert^\alpha}{\vert 1-t\vert^{2+s} } \dd t &=  \int_{(0,+\infty)\setminus (1-\varepsilon,1+\varepsilon)}\frac{1-  t^\alpha}{\vert 1-t\vert^{2+s} } \dd t+\int_{-\infty}^0\frac{1- (-t)^\alpha}{( 1-t)^{2+s} } \dd t \\
        &= \int_{(0,+\infty)\setminus (1-\varepsilon,1+\varepsilon)}\frac{1-  t^\alpha}{\vert 1-t\vert^{2+s} } \dd t+\int_0^{+\infty}\frac{1- t^\alpha}{( 1+t)^{2+s} } \dd t \\
        &= I_1^{(\varepsilon)}+I_2^{(\varepsilon)}+I_3^{(\varepsilon)},
    \end{align*} where \begin{align*}
         I_1^{(\varepsilon)}& \coloneqq \int_0^{1-\varepsilon}  \big ( 1- t^\alpha \big ) \bigg [ \frac1{(1-t)^{2+s} } + \frac1{( 1+t)^{2+s} } \bigg] \dd t, \\
        I_2^{(\varepsilon)} & \coloneqq \int_{1+\varepsilon}^{+\infty} \big ( 1- t^\alpha \big ) \bigg [ \frac1{(t-1)^{2+s} } + \frac1{( 1+t)^{2+s} } \bigg] \dd t, \\
        I_3^{(\varepsilon)} & \coloneqq \int_{1-\varepsilon}^{1+\varepsilon}\frac{1- t^\alpha}{( 1+t)^{2+s} } \dd t.
    \end{align*}

    For~\(I_3^{(\varepsilon)}\), we have that \begin{align*}
        \big \vert I_3^{(\varepsilon)} \big \vert \leq \int_{1-\varepsilon}^{1+\varepsilon} \big \vert {1- t^\alpha} \big \vert  \dd t \leq C \varepsilon^2, 
    \end{align*}
	so~$\lim_{\varepsilon\to0^+} I_3^{(\varepsilon)} =0$.

     For~\(I_2^{(\varepsilon)}\), we make the change of variable~\(t\to t^{-1}\) to obtain \begin{align*}
        I_2^{(\varepsilon)} &= \int_0^{\frac 1 {1+\varepsilon}} \big ( 1- t^{-\alpha} \big ) \bigg [ \frac1{(t^{-1}-1 )^{2+s} } + \frac1{( 1+t^{-1} )^{2+s} } \bigg] \frac{\dd t}{t^2} \\
        &= -\int_0^{\frac 1 {1+\varepsilon}} t^{s-\alpha}\big ( 1-t^\alpha \big ) \bigg [ \frac1{(1-t )^{2+s} } + \frac1{( t+1 )^{2+s} } \bigg] \dd t .
    \end{align*} Hence, if \begin{align*}
        I_4^{(\varepsilon)} \coloneqq \int_{1-\varepsilon}^{\frac 1 {1+\varepsilon}} t^{s-\alpha}\big ( 1-t^\alpha \big ) \bigg [  \frac1{(1-t )^{2+s} } + \frac1{( t+1 )^{2+s} }\bigg ] \dd t,
    \end{align*} it follows that \begin{align*}
        I_1^{(\varepsilon)}+I_2^{(\varepsilon)} &=\int_0^{1-\varepsilon}  \big ( 1- t^\alpha\big )\big (1 -t^{s-\alpha} \big ) \bigg [  \frac1{(1-t )^{2+s} } + \frac1{( t+1 )^{2+s} }\bigg ]\dd t  -I_{4}^{(\varepsilon)}.
    \end{align*}
	Note that
	$$
        0\leq I_{4}^{(\varepsilon)}
		\le C \int_{1-\varepsilon}^{\frac 1 {1+\varepsilon}}  \frac{ 1-t^\alpha}{(1-t )^{2+s} } \dd t \le C \int_{1-\varepsilon}^{\frac 1 {1+\varepsilon}}  \frac{ \dd t}{(1-t )^{1+s} } \dd t \le C \varepsilon^{1 - s}.
	$$

    Collecting all the estimates above, we have that \begin{align*}
         \PV \int_{\R}\frac{1- \vert t\vert^\alpha}{\vert 1-t\vert^{2+s} } \dd t &= \lim_{\varepsilon\to 0^+} \Big ({I_1^{(\varepsilon)}+I_2^{(\varepsilon)}+I_3^{(\varepsilon)}} \Big )\\
         &= \lim_{\varepsilon \to 0^+} \int_0^{1-\varepsilon}  \big ( 1- t^\alpha\big )\big (1 -t^{s-\alpha} \big )\bigg [ \frac1{(1-t )^{2+s} } + \frac1{( t+1 )^{2+s} } \bigg ] \dd t,
    \end{align*}as required. 
\end{proof}

Thanks to the previous two lemmas, we can now address the proof of Proposition~\ref{prop:barrier}.

\begin{proof}[Proof of Proposition~\ref{prop:barrier}]
    Let~$x \in \partial F_{n, \alpha} \setminus B_1$. By symmetry, we may assume without loss of generality that~$x' = r e_1$ with~$r > 0$. Notice that~$x \notin B_1$ actually implies that~$r \ge 2^{- \frac{1}{2 \alpha}}$. Recalling definition~\eqref{Fdef} and representation~\eqref{s-curvofgraphs}, we write
    $$
       \mathrm H_{s,F_{n, \alpha}}(x) = 2 \, \PV \int_{\R^n} G_s \bigg (\frac{ r^\alpha-\vert y'\vert^\alpha}{\vert r e_1 - y' \vert } \bigg ) \frac{\dd y' }{\vert r e_1 - y'\vert^{n+s}}.
    $$
	Making the change of variable~\( y' \to r y' \) gives
    $$
        \mathrm H_{s,F_{n, \alpha}}(x)=2 r^{-s} \, \PV \int_{\R^n} G_s \bigg (r^{\alpha-1}\frac{1- \vert y'\vert^\alpha}{\vert e_1 - y' \vert } \bigg ) \frac{\dd y'}{\vert e_1 - y'\vert^{n+s}},
    $$
	which, recalling the definition of~$\widetilde{G}_s$ given in Lemma~\ref{Gslem}, we rewrite as
    \begin{equation} \label{HsFasGtilde}
        \mathrm H_{s,F_{n, \alpha}}(x) = 2 \beta_{n, s, \alpha} \, r^{\alpha - 1 - s} + 2 r^{-s} \, \PV \int_{\R^n} \widetilde{G}_s \bigg (r^{\alpha-1}\frac{1- \vert y' \vert^\alpha}{\vert e_1 - y' \vert } \bigg ) \frac{\dd y'}{\vert e_1 - y' \vert^{n+s}},
    \end{equation}
    with
    $$
        \beta_{n, s, \alpha} \coloneqq \PV \int_{\R^n}\frac{1- \vert y'\vert^\alpha}{\vert e_1 - y'\vert^{n+1+s} } \dd y'.
    $$
    
    We now proceed to estimate the second summand on the right-hand side of~\eqref{HsFasGtilde}. First of all, a simple computation yields that
    $$
        -\alpha \vert e_1 - y' \vert  \leq 1-\vert y' \vert^\alpha \leq \vert e_1 - y' \vert \quad \text{for all } y' \in \R^n,
    $$
    and thus
    $$
        \left| \frac{1 - \vert y'\vert^\alpha}{\vert e_1 - y'\vert } \right| \le 1 \quad \mbox{for all } y' \in \R^n.
    $$
    By applying estimate~\eqref{Gsest} of Lemma~\ref{Gslem} with~$a = r^{\alpha-1}\frac{1- \vert y'\vert^\alpha}{\vert e_1 - y'\vert }$ and~$b = 0$, we see that
    \begin{equation} \label{outerest}
    \begin{aligned}
        & \Bigg| {\int_{\R^n \setminus B_{\frac{1}{2}}'(e_1)} \widetilde{G}_s \bigg (r^{\alpha-1}\frac{1- \vert y'\vert^\alpha}{\vert e_1 - y' \vert } \bigg ) \frac{\dd y'}{\vert e_1 - y'\vert^{n+s}}} \Bigg|\\
		& \hspace{70pt} \le \int_{\R^n \setminus B_{\frac{1}{2}}'(e_1)} \left| \widetilde{G}_s \bigg (r^{\alpha-1}\frac{1- \vert y'\vert^\alpha}{\vert e_1 - y' \vert } \bigg ) \right| \frac{\dd y'}{\vert e_1 - y' \vert^{n+s}} \\
        & \hspace{70pt} \le \frac{n + 2}{2} \, r^{3(\alpha-1)} \int_{\R^n \setminus B_{\frac{1}{2}}'} \frac{\dd z'}{\vert z' \vert^{n + s} } \le \frac{(n + 2) \mathcal{H}^{n - 1} \big( {\partial B_1'} \big)}{s} \, r^{3(\alpha-1)},
    \end{aligned}
    \end{equation}
	where, here and in the remainder of the paper,~$B_r'(x_0')$ indicates the open~$n$-dimensional ball of radius~$r > 0$ centred at a point~$x_0' \in \R^n$, while~$B_r' \coloneqq B_r'(0)$. The bound for the integral over~$B_{\frac{1}{2}}'(e_1)$ is slightly more involved. We begin by observing that, by symmetry,
    $$
        \PV \int_{B_{\frac{1}{2}}'(e_1)} \widetilde{G}_s \bigg (\alpha r^{\alpha-1}\frac{1 - y_1}{\vert e_1 - y'\vert } \bigg ) \frac{\dd y'}{\vert e_1 - y' \vert^{n+s}} = 0.
    $$
    Hence, by subtracting this term we may write
    \begin{align*}
        & \PV \int_{B_{\frac{1}{2}}'(e_1)} \widetilde{G}_s \bigg (r^{\alpha-1}\frac{1- \vert y'\vert^\alpha}{\vert e_1 - y' \vert } \bigg ) \frac{\dd y'}{\vert e_1 - y'\vert^{n+s}} \\
        & \hspace{50pt} = \int_{B_{\frac{1}{2}}'(e_1)} \left\{ \widetilde{G}_s \bigg (r^{\alpha-1}\frac{1- \vert y'\vert^\alpha}{\vert e_1 - y' \vert } \bigg ) - \widetilde{G}_s \bigg (\alpha r^{\alpha-1}\frac{1 - y_1}{\vert e_1 - y'\vert } \bigg ) \right\} \frac{\dd y'}{\vert e_1 - y' \vert^{n+s}},
    \end{align*}
    from which, using~\eqref{Gsest} with~$a = r^{\alpha-1}\frac{1- \vert y'\vert^\alpha}{\vert e_1 - y' \vert }$ and~$b = \alpha r^{\alpha-1}\frac{1 - y_1}{\vert e_1 - y'\vert }$, it follows that
    \begin{align*}
		& \Bigg| {\PV \int_{B_{\frac{1}{2}}'(e_1)} \widetilde{G}_s \bigg (r^{\alpha-1}\frac{1- \vert y'\vert^\alpha}{\vert e_1 - y' \vert } \bigg ) \frac{\dd y'}{\vert e_1 - y'\vert^{n+s}}} \Bigg| \\
		& \hspace{90pt} \le \frac{n + 2}{2} \, r^{3(\alpha - 1)} \int_{B_{\frac{1}{2}}'(e_1)} \frac{\big| {1- \vert y'\vert^\alpha - \alpha (1 - y_1)} \big|}{\vert e_1-y'\vert^{n + 1 +s}} \dd y'.
    \end{align*}
    Since
	$$
        \big| {1- \vert y'\vert^\alpha - \alpha (1 - y_1)} \big| \le C_n |y' - e_1|^2 \quad \mbox{for all } y' \in B_{\frac{1}{2}}(e_1),
    $$
    for some dimensional constant~$C_n > 0$, we infer that
    \begin{align*}
		\Bigg| {\PV \int_{B_{\frac{1}{2}}'(e_1)} \widetilde{G}_s \bigg (r^{\alpha-1}\frac{1- \vert y'\vert^\alpha}{\vert e_1 - y' \vert } \bigg ) \frac{\dd y'}{\vert e_1 - y'\vert^{n+s}}} \Bigg| & \le C_n r^{3(\alpha - 1)} \int_{B_{\frac{1}{2}}'(e_1)} \frac{\dd y'}{\vert e_1-y'\vert^{n - 1 +s}} \\
		& \le \frac{C_n}{1 - s} \, r^{3(\alpha - 1)},
	\end{align*}
    for some possibly larger~$C_n$. By combining this estimate with~\eqref{outerest}, identity~\eqref{HsFasGtilde} yields
    \begin{equation} \label{barrierfinalineq}
    \left| \mathrm H_{s,F_{n, \alpha}}(x) - 2 \beta_{n, s, \alpha} \, r^{\alpha - 1 - s} \right| \le \frac{C_n}{s(1 - s)} \, r^{3(\alpha - 1) - s} \quad \mbox{for every } r \ge 2^{- \frac{1}{2 \alpha}}.
    \end{equation}

	Notice that from this it immediately follows
	$$
		\left| \mathrm H_{s,F_{n, \alpha}}(x) \right| \le C \, |x'|^{\alpha - 1 - s} \quad \mbox{for every } x \in \partial F_{n, \alpha} \setminus B_1,
	$$
	for some constant~$C > 0$, depending only on~$n$,~$s$, and~$\alpha$, which readily gives claim~\ref{barriercaseii}. When~$n = 1$ and~$\alpha \in (0, s)$, Lemma~\ref{GO1TWVl8} ensures that~$\beta_{1, s, \alpha} > 0$. Therefore, inequality~\eqref{barrierfinalineq} yields
    $$
        \mathrm H_{s,F_{1, \alpha}}(x) \ge \left( 2 \beta_{1, s, \alpha} - \frac{C_1}{s(1 - s)} \, \frac{1}{r^{2(1 - \alpha)}} \right) r^{\alpha - 1 - s} \ge \beta_{1, s, \alpha} \, r^{\alpha - 1 - s},
    $$
    provided~$r$ is large enough, in dependence of~$s$ and~$\alpha$. Hence, claim~\ref{barriercasei} is also established.
\end{proof}

\subsection{Density estimates} \label{subsec:density}
We include here a couple of ``thick'' density estimates that we will use in the proof of Theorem~\ref{thm:s-halfspace} presented in Section~\ref{sec:ndproof}. Note that these estimates are well-known for minimisers of the fractional perimeter---see~\cite[Theorem~4.1]{CRS10}. Their novelty thus resides in the fact that they hold for more general, not necessarily minimising, nonlocal minimal surfaces.

Our first result is for sets that are stable for the~$s$-perimeter. Recall the following definition of stability from~\cite{CSV19,CCS20}.

\begin{defn} \label{def:stability}
A set~$E \subset \R^{n + 1}$ is said to be~\emph{stable} for the~$s$-perimeter in a bounded open set~$\Omega \subset \R^{n + 1}$ if~$\Per_s(E; \Omega) < +\infty$ and, given any vector field~$X = X(x, t) \in C^\infty \big( \R^{n + 1} \times (-1, 1); \R^{n + 1} \big)$ with~$\supp(X) \subset \Omega \times (- 1, 1)$, we have that
$$
\liminf_{t \rightarrow 0} \frac{1}{t^2} \Big( {\Per_s \! \big( {\Phi_X^t(E) \cap E; \Omega} \big) -  \Per_s(E; \Omega)} \Big) \ge 0
$$
and
$$
\liminf_{t \rightarrow 0} \frac{1}{t^2} \Big( {\Per_s \! \big( {\Phi_X^t(E) \cup E; \Omega} \big) -  \Per_s(E; \Omega)} \Big) \ge 0,
$$
where~$\Phi_X^t$ is the integral flow of~$X$ at time~$t$. When no reference to a container~$\Omega$ is made, we simply mean that~$E$ is stable in every bounded open set of~$\R^{n + 1}$.
\end{defn}

Note that, given a measurable~$F \subset \R^{n + 1}$ and an open set~$\Omega \subset \R^{n + 1}$, we indicate with~$\Per_s(F; \Omega)$ the fractional~$s$-perimeter of~$F$ inside~$\Omega$, defined by
$$
\Per_s(F; \Omega) \coloneqq \frac{1}{2} \iint_{Q(\Omega)} \frac{|\chi_F(x) - \chi_F(y)|}{|x - y|^{n + 1 + s}} \dd x \dd y,
$$
with
$$
Q(\Omega) \coloneqq \big( {\Omega \times \Omega} \big) \cup \big( {\Omega \times (\R^{n + 1} \setminus \Omega)} \big) \cup \big( {(\R^{n + 1} \setminus \Omega) \times \Omega} \big).
$$

We point out that Definition~\ref{def:stability} is not the only notion of~$s$-stability considered in the literature---see~\cite[Definition~2.2]{CCS20} for a weaker and perhaps more natural definition, as well as~\cite[Section~2]{CCS20} and~\cite[Section~1.2]{C25} for discussions on their relationship. Here, we consider stable sets in the sense of Definition~\ref{def:stability} since we need them to satisfy the perimeter estimates established in~\cite{CSV19}. These, alongside the fractional perimeter bounds of~\cite{T26}, are the main tools needed to prove density estimates for~$s$-stable sets, as stated in the following result.

\begin{prop} \label{prop:densestforstable}
Let~$n \in \N$ and~$s \in (0, 1)$. Let~$E \subset \R^{n + 1}$ be a stable set for the~$s$-perimeter. Then,
$$
\big| {E \cap B_\rho(q)} \big| \ge c_\flat \, \rho^{n + 1} \quad \mbox{and} \quad \big| {B_\rho(q) \setminus E} \big| \ge c_\flat \, \rho^{n + 1} \quad \mbox{for every } q \in \partial E \mbox{ and } \rho > 0,
$$
for some constant~$c_\flat > 0$ depending only on~$n$ and~$s$.
\end{prop}
\begin{proof}
First of all, up to a translation, we may suppose that~$q = 0$. Moreover, within this proof we shall indicate with~$C$ a constant larger than~$1$, whose value may change from line to line and depends only on~$n$ and~$s$.

By~\cite[Corollary~1.8]{CSV19}, we know that
\begin{equation} \label{upperPerest}
\Per(E; B_\rho) \le C \rho^{n} \quad \mbox{for every } \rho > 0,
\end{equation}
where~$\Per(E; \Omega)$ denotes the classical perimeter of~$E$ inside an open set~$\Omega \subset \R^{n + 1}$. In consequence of~\eqref{upperPerest}, the set~$E$ has locally finite perimeter.  We are then free to assume that~$0 \in \partial^\ast \! E$---indeed, the general case follows by approximation, exploiting the fact that the reduced boundary~$\partial^\ast \! E$ is dense in~$\partial E$. In light of these considerations, we may argue as in the proof of~\cite[Theorem~1.1]{T26} and deduce that
\begin{equation} \label{lowerPersest}
\Per_s(E; B_R) \ge C^{-1} R^{n + 1 - s} \quad \mbox{for every } R > 0.
\end{equation}

Recall now the interpolation inequality of~\cite[Corollary~3.2]{T26}, which gives that
$$
\Per_s \! \Big( {E; B_{\frac{\varepsilon^{1/s} \rho}{2}}} \Big) \le C \Big( {\min \big\{ {|E \cap B_\rho|, |B_\rho \setminus E|} \big\}^{1 - s} \Per(E; B_\rho)^s + \varepsilon^{\frac{n + 1}{s}} \rho^{n + 1 - s}} \Big),
$$
for all~$\varepsilon \in \left( 0, \frac{1}{4} \right]$. By combining this with~\eqref{upperPerest} and~\eqref{lowerPersest} (with~$R = \frac{\varepsilon^{1/s} \rho}{2}$), we obtain
$$
\rho^{n + 1 - s} \le C \Big( {\varepsilon^{- \frac{n + 1 - s}{s}} \min \big\{ {|E \cap B_\rho|, |B_\rho \setminus E|} \big\}^{1 - s} \rho^{s n} + \varepsilon \rho^{n + 1 - s}} \Big).
$$
By taking~$\varepsilon$ suitably small, in dependence of~$n$ and~$s$ alone, we can reabsorb to the left-hand side the last term on the right and easily obtain the claims of the proposition.
\end{proof}

In our second result, we establish density estimates for general nonlocal minimal surfaces (not necessarily minimisers, nor stable sets), but only at points of their boundaries that can be touched by a ball and only at scales smaller than the radius of said ball. Its precise statement is as follows.

\begin{prop} \label{prop:densestfors-stat}
Let~$n \in \N$ and~$s \in (0, 1)$. Let~$E \subset \R^{n + 1}$ be a measurable set such that~$\mathrm{H}_{s, E} \ge 0$ in the viscosity sense at some point~$q \in \partial E$. Assume that there exists a ball~$B$ of radius~$r > 0$ such that~$B \subset E$ and~$q \in \partial B$. Then,
$$
\big| {E \cap B_\rho(q)} \big| \ge c_\sharp \, \rho^{n + 1} \quad \mbox{and} \quad \big| {B_\rho(q) \setminus E} \big| \ge c_\sharp \, \rho^{n + 1} \quad \mbox{for every } \rho \in (0, r],
$$
for some constant~$c_\sharp > 0$ depending only on~$n$ and~$s$.
\end{prop}
Our argument is reminiscent of the one used for the ``cheap'' proof of the weak Harnack inequality for equations driven by the fractional Laplacian---see, e.g.,~\cite[Theorem~2.2]{RS19} or~\cite[Proposition~2.3.4]{BV16}.

\begin{proof}[Proof of Proposition~\ref{prop:densestfors-stat}]

Up to a translation, a rotation, and a dilation of~$\R^{n + 1}$, we may assume that~$q = 0$,~$r = 1$, and~$B = B_1(e_{n + 1})$. Let then~$\rho \in (0, \rho_0]$, with~$\rho_0 \in (0, 1]$ soon to be chosen small but in dependence of~$n$ and~$s$ only. Thanks to Proposition~\ref{prop:visc-equiv} (and the discussion preceding its statement), we know that~$\mathrm{H}_{s, E}(0)$ is well-defined, finite, and non-negative. Hence,
$$
0 \le \mathrm{H}_{s, E}(0) = \mathrm{H}_{s, E}(0) - \mathrm{H}_{s, B}(0) + \mathrm{H}_{s, B}(0) = - 2 \int_{E \setminus B} \frac{\dd y}{|y|^{n + 1 + s}} + \mathrm{H}_{s, B}(0),
$$
that is,
$$
\gamma_{n, s} \ge \int_{E \setminus B_1(e_{n + 1})} \frac{\dd y}{|y|^{n + 1 + s}},
$$
where~$\gamma_{n, s} \coloneqq \frac{1}{2} \, \mathrm{H}_{s, B}(0)$, positive constant depending only on~$n$ and~$s$. From here, we simply estimate
$$
\gamma_{n, s} \ge \int_{E \cap \big( {B_\rho \setminus B_1(e_{n + 1})} \big)} \frac{\dd y}{|y|^{n + 1 + s}} \ge \frac{\big| {E \cap \big( {B_\rho \setminus B_1(e_{n + 1})} \big)} \big|}{\rho^{n + 1 + s}}.
$$
This and the fact that~$B_1(e_{n + 1}) \subset \R^{n + 1}_+$ give that
\begin{align*}
\big| {E \cap B_\rho} \big| & = \big| {E \cap \big( {B_\rho \setminus B_1(e_{n + 1})} \big)} \big| + \big| {E \cap B_\rho \cap B_1(e_{n + 1})} \big| \\
& \le \gamma_{n, s} \, \rho^{n + 1 + s} + \big| {B_\rho \cap \R^{n + 1}_+} \big| = \Big( {\gamma_{n, s} \, \rho^s + \frac{1}{2} \, |B_1|} \Big) \rho^{n + 1} \le \frac{3}{4} \, |B_1| \, \rho^{n + 1},
\end{align*}
provided we take~$\rho_0 \coloneqq \min \Big\{ {\Big( {\frac{|B_1|}{4 \gamma_{n, s}}} \Big)^{\! \frac{1}{s}}, 1} \Big\}$. Consequently,
$$
\big| {B_\rho \setminus E} \big| = \big| {B_\rho} \big| - \big| {E \cap B_\rho} \big| \ge \frac{1}{4} \, |B_1| \, \rho^{n + 1} \quad \mbox{for every } \rho \in (0, \rho_0].
$$
Note that, up to making the constant in front of~$\rho^{n + 1}$ smaller, the same estimate also holds for~$\rho \in [\rho_0, 1]$, since
$$
\big| {B_\rho \setminus E} \big| \ge \big| {B_{\rho_0} \setminus E} \big| \ge \frac{1}{4} \, |B_1| \, \rho_0^{n + 1} \ge \Big( {\frac{1}{4} \, |B_1| \, \rho_0^{n + 1}} \Big) \rho^{n + 1} \quad \mbox{for every } \rho \in (\rho_0, 1].
$$
Finally, the bound from below on the measure of~$E \cap B_\rho$ follows trivially from the inclusion~$B_1(e_{n + 1}) \subset E$.
\end{proof}

\section{Proof of Theorem~\texorpdfstring{\ref{thm:s-halfspace}}{\ref{thm:s-halfspace}} for~\texorpdfstring{$n = 1$}{n = 1}} \label{sec:1dproof}

\noindent
We provide here the proof of the following statement, which yields in particular Theorem~\ref{thm:s-halfspace} in the case~$n = 1$.

\begin{thm} \label{thm:techthm1}
    Let~\(s\in (0,1)\). Let~\(E\subset \R^2\) be a measurable set contained in a half-plane and that satisfies~\(\mathrm H_{s,E}\leq 0\) in the viscosity sense. Then,~\(E\) is a half-plane. 
\end{thm}

\begin{proof}
    For the sake of contradiction, assume that \(E\) is not a half-plane. By rotating and translating, we may assume without loss of generality that~\(E\subset \R^2_- \coloneqq \R \times (-\infty, 0)\). By further translating vertically, we may also assume that
    \begin{equation} \label{distEhalf=0}
        \dist ( E, \partial \R^2_-)=0.
    \end{equation}
    
    Observe that~\(\partial E \cap \partial \R^2_- = \varnothing\), since, otherwise, the strong comparison principle of Proposition~\ref{prop:compprinc} would imply that \(E=\R^2_-\). As a result, there exists a small~$r \in (0, 1)$ such that~$[- r, r] \times [- 2r, 0]$ has empty intersection with~$E$.

    Let now~$\alpha \in (0, 1)$ be fixed,~$F = F_{1, \alpha}$ be the set defined in~\eqref{Fdef}, and~$F_\varepsilon \coloneqq \varepsilon F$ be its rescaling of a factor~$\varepsilon > 0$. Note that
    \begin{equation} \label{Fepsexpr}
        F_\varepsilon = \Big\{ {(x_1, x_2) \in \R^2 \text{ s.t. } x_2< \varepsilon^{1 - \alpha} \vert x_1\vert^\alpha} \Big\}.
    \end{equation}
    By Proposition~\ref{prop:barrier}\ref{barriercasei} and an immediate scaling argument, we know that $\mathrm H_{s, F_\varepsilon}(x) > 0$ for all~$x \in \partial F_{\varepsilon} \setminus B_{\varepsilon R}$, for some constant~$R \ge 1$ independent of~$\varepsilon$. Thus,
    \begin{equation} \label{Fepssuper}
        \mathrm H_{s, F_\varepsilon}(x) > 0 \quad \mbox{for every } x \in \partial F_\varepsilon \setminus B_r \mbox{ and } \varepsilon \in \left( 0, \frac{r}{R} \right]. 
    \end{equation}
    Consider then the vertical translation~$E_\delta \coloneqq E + (0, \delta)$, for~$\delta \in (0, r]$, and observe that
    \begin{equation} \label{ErinsideL}
        E_\delta \subset \big( {\R \times (-\infty, \delta)} \big) \setminus \big( {[-r, r] \times [- r, \delta]} \big).
    \end{equation}
    By the translation invariance of the~$s$-mean curvature, we clearly have that
    \begin{equation} \label{Edeltasubsol}
        \mathrm H_{s, E_\delta} \le 0 \quad \mbox{at every } x \in \partial E_\delta \mbox{ in the viscosity sense}.
    \end{equation}
    Moreover, it easily follows from~\eqref{ErinsideL} that
    \begin{equation} \label{ErsubsetofFr}
        E_\delta \subset F_\varepsilon \quad \mbox{for every } \varepsilon \ge \frac{r}{R},
    \end{equation}
    provided~$\delta$ is chosen smaller than~$r R^{\alpha - 1}$.
    
    Define
    $$
        \varepsilon_0 \coloneqq \inf \Big\{ {\varepsilon \in (0, 1) \text{ s.t. } E_\delta \subset F_\varepsilon} \Big\}.
    $$
    In view of~\eqref{ErsubsetofFr}, we have that~$\varepsilon_0$ is well-defined and belongs to~$\left[ 0, \frac{r}{R} \right]$. Our goal is to show that~$\varepsilon_0 = 0$. To this aim, we argue by contradiction and suppose~$\varepsilon_0 > 0$. From this assumption, it immediately follows that there exists a sequence of points~$p^{(k)} \in E_\delta \setminus F_{\varepsilon_0 - \frac{1}{k}}$, for~$k \in \N$ large. Note that this sequence is bounded, as, indeed, the combination of inclusion~\eqref{ErinsideL} with the definition of the sets~$F_{\varepsilon_0 - \frac{1}{k}}$ gives that~$|p^{(k)}| \le 4 \varepsilon_0^{1 - \frac{1}{\alpha}}$ for all large enough~$k$'s. Therefore, up to a subsequence,~$\{ p^{(k)} \}$ converges to some limit point~$p \in \overline{E_\delta} \setminus F_{\varepsilon_0}$. Since, by definition of~$\varepsilon_0$, it holds~$\overline{E_\delta} \subset \overline{F_{\varepsilon_0}}$, we infer that~$p$ belongs to the boundaries of both~$E_\delta$ and~$F_{\varepsilon_0}$. Furthermore, inclusion~\eqref{ErinsideL} and the fact that~$p \in \overline{E_\delta}$ yield that~$|p| \ge r$.
    
    All in all, we have found a point~$p \in \big( {\partial E_\delta \cap \partial F_{\varepsilon_0}} \big) \setminus B_r$ at which~$F_{\varepsilon_0}$ touches~$E_\delta$ from the outside and where, thanks to~\eqref{Fepssuper} and~\eqref{Edeltasubsol}, the inequalities~$\mathrm H_{s, F_{\varepsilon_0}} > 0$ and~$\mathrm H_{s, E_\delta} \le 0$ hold, respectively, in the pointwise and viscosity sense. This is, however, impossible, as can be seen by employing the comparison principle of Proposition~\ref{prop:compprinc} or by directly applying Definition~\ref{def:viscsol}. Anyway, we conclude that~$\varepsilon_0 = 0$, which, recalling its definition and~\eqref{Fepsexpr}, gives that
    $$
        E_\delta \subset \bigcap_{\varepsilon \in (0, 1)} F_\varepsilon \subset \overline{\R_-^2}.
    $$
    Going back to~$E$, this means that~$E \subset \overline{\R_-^2} - (0, \delta) = \R \times (-\infty, -\delta]$. Since this contradicts~\eqref{distEhalf=0}, we conclude that~$E$ is a half-space, thereby ending the proof.
\end{proof}

\section{Proof of Theorem~\texorpdfstring{\ref{thm:s-halfspace}}{\ref{thm:s-halfspace}} for a general~\texorpdfstring{$n \ge 1$}{n ≥ 1}} \label{sec:ndproof}

\noindent
In this section, we present the proof of Theorem~\ref{thm:s-halfspace} in full generality. Its claims will actually come as a corollary of the following more technical statement.

\begin{thm} \label{thm:techthm2}
	Let~$n \in \N$ and~$s \in (0, 1)$. Let~\(E\subset \R^{n + 1}\) be a measurable set contained in a half-plane and that satisfies~\(\mathrm H_{s,E} = 0\) in the viscosity sense. Assume that either of the following assumptions holds true:
	\begin{enumerate}[leftmargin=*,label=($\alph*$)]
	\item \label{tech-a} the complement of~$E$ enjoys uniform thick density estimates around every point of its boundary and at all scales, i.e., there exists a constant~$c_\star > 0$ such that
	$$
		\big| {B_\rho(q) \setminus E} \big| \ge c_\star \rho^{n + 1} \quad \mbox{for all } q \in \partial E \mbox{ and } \rho > 0,
	$$
	or
	\item \label{tech-b} $E$ has boundary of class~$C^1$ outside of a compact set.
	\end{enumerate}
	Then,~$E$ is a half-space.
\end{thm}

Observe that the hypotheses of Theorem~\ref{thm:techthm2} are satisfied in particular by smooth nonlocal minimal surfaces---for which~\ref{tech-b} clearly holds true. Thus, Theorem~\ref{thm:s-halfspace} follows. On the other hand, its hypotheses are also fulfilled by stable sets for the fractional perimeter---since the density estimates of~\ref{tech-a} are valid for such sets, thanks to Proposition~\ref{prop:densestforstable}

\begin{proof}[Proof of Theorem~\ref{thm:techthm2}]
For readability purposes, we divide the argument into several steps. Steps~1-4 are carried out under the sole assumptions that~$E$ is contained in a half-space and has vanishing~$s$-mean curvature in the viscosity sense. In the fifth step, we address the proof of a claim made in Step 4: it is split into Steps~5a and~5b, each establishing the claim under the extra hypotheses~\ref{tech-a} or~\ref{tech-b}. The argument is closed by the final Step~6, common to both cases.

Steps~1,~2, and~6 follow a path similar to the proof of Theorem~\ref{thm:techthm1}. The three additional steps~3,~4, and~5 involve further ideas, that are not required in the two-dimensional case, owing to the presence of a simply constructed supersolution, as per Proposition~\ref{prop:barrier}\ref{barriercasei}.

We now begin the actual proof of Theorem~\ref{thm:techthm2}.
\begin{itemize}[leftmargin=0pt]
\item[] \textbf{Step 1. Setup of the argument.} 
We argue by contradiction and suppose that
\begin{equation} \label{Enothalfspace}
E \mbox{ is not a half-space}.
\end{equation}
After a rotation and a translation, we may assume that~$E \subset \R^{n + 1}_- \coloneqq \R^n \times (-\infty, 0)$ and that
\begin{equation} \label{Eclosetohorhyp}
\dist \Big( {E, \partial \R^{n + 1}_-} \Big) = 0.
\end{equation}
Notice that
\begin{equation} \label{Edetached}
\partial E \cap \partial \R^{n + 1}_- = \varnothing,
\end{equation}
since, otherwise, the strong comparison principle of Proposition~\ref{prop:compprinc} would yield that~$E = \R^{n + 1}_-$. After a dilation, we may then assume that the cylinder~$\overline{B_2'} \times [-2, 0]$ has empty intersection with~$E$.

For~$\alpha \in (0, 1)$ fixed, consider now the set~$F = F_{n, \alpha}$ defined in~\eqref{Fdef}, along with its rescaling
\begin{equation} \label{Fepsdef}
F_\varepsilon \coloneqq \varepsilon F = \Big\{ {x \in \R^{n + 1}} \mbox{ s.t. } x_{n + 1} < \varepsilon^{1 - \alpha} |x'|^\alpha \Big\}
\end{equation}
of a factor~$\varepsilon > 0$. Thanks to Proposition~\ref{prop:barrier}\ref{barriercaseii} and the scaling properties of the~$s$-mean curvature, we have that
\begin{equation} \label{HsFepssmall}
\left| \mathrm H_{s, F_\varepsilon}(x) \right| \le C_0 \, \frac{\varepsilon^{1 - \alpha}}{|x|^{s + 1 - \alpha}} \quad \mbox{for all } x \in \partial F_{\varepsilon} \setminus B_\varepsilon,
\end{equation}
for some constant~$C_0 \ge 1$ depending only on~$n$,~$s$, and~$\alpha$. Then, for~$j \in \mathbb{N}$, let
\begin{equation} \label{Ejdef}
E_j \coloneqq E + \frac{1}{j} \, e_{n + 1}.
\end{equation}
Clearly,~$E_j$ has zero~$s$-mean curvature in the viscosity sense. Moreover, it satisfies
\begin{equation} \label{EjinLj}
E_j \subset \bigg( {\R^n \times \Big( {-\infty, \frac{1}{j}} \Big)} \bigg) \setminus \bigg( {B_1' \times \Big[ {-1, \frac{1}{j}} \Big)} \bigg)
\end{equation}
and therefore
\begin{equation} \label{EjinFepsilon}
E_j \subset F_\varepsilon \quad \mbox{for every } \varepsilon \ge j^{- \frac{1}{1 - \alpha}}.
\end{equation}

Define
\begin{equation} \label{epsjdef}
\varepsilon_j \coloneqq \inf \Big\{ {\varepsilon > 0 \text{ s.t. } E_j \subset F_\varepsilon} \Big\}.
\end{equation}
Thanks to~\eqref{EjinFepsilon}, we have that~$\varepsilon_j$ is a well-defined non-negative real number such that
\begin{equation} \label{epsilonsmallerthan1overj}
\varepsilon_j \le j^{- \frac{1}{1 - \alpha}}.
\end{equation}
In particular,
\begin{equation} \label{epsjgoestozero}
\lim_{j \rightarrow +\infty} \varepsilon_j = 0.
\end{equation}
Our goal is to show that
\begin{equation} \label{fundamentalclaim}
\mbox{there exists } j_0 \in \N \mbox{ for which } \varepsilon_{j_0} = 0.
\end{equation}

\item[] \textbf{Step 2. Finding a diverging sequence of touching points.}
In order to prove~\eqref{fundamentalclaim}, we argue by contradiction and suppose instead that~$\varepsilon_j > 0$ for every~$j \in \N$. From this assumption, it immediately follows that, for any fixed~$j \in \N$, there exists a sequence of points~$\{ p_{k, j} \}_{k \in \N}$ with~$p_{k, j} \in E_j \setminus F_{\varepsilon_j - \frac{1}{k}}$, for~$k$ large. Note that this sequence is bounded as, indeed, the combination of inclusion~\eqref{EjinLj} with the definition of the sets~$F_{\varepsilon_j - \frac{1}{k}}$ gives that~$|p_{k, j}| \le 4 \varepsilon_j^{1 - \frac{1}{\alpha}}$ for all large enough~$k$'s. Therefore, up to a subsequence,~$\{ p_{k, j} \}_{k \in \N}$ converges to some limit point~$p_j \in \overline{E_j} \setminus F_{\varepsilon_j}$. Since, by definition of~$\varepsilon_j$, it holds~$\overline{E_j} \subset \overline{F_{\varepsilon_j}}$, we infer that~$p_j$ belongs to~$\partial E_j \cap \partial F_{\varepsilon_j}$. Furthermore, from inclusion~\eqref{EjinLj} it immediately follows that~$|p_j'| \ge 1$ and
\begin{equation} \label{pjverticallysmall}
(p_j)_{n + 1} \in \Big( {0, \frac{1}{j}} \Big].
\end{equation}

We claim that
\begin{equation} \label{pjunbounded}
\mbox{the sequence } \{ p_j \}_{j \in \N} \mbox{ is unbounded}.
\end{equation}
Arguing by contradiction, we assume that it is bounded and thus, it converges to some point~$p \in \R^{n + 1}$, up to a subsequence. Recalling~\eqref{epsjgoestozero}, definitions~\eqref{Fepsdef} and~\eqref{Ejdef}, as well as the fact that~$p_j \in \partial E_j \cap \partial F_{\varepsilon_j}$ for every~$j \in \N$, it is clear that~$p \in \partial E \cap \partial \R^{n + 1}_-$. As this contradicts~\eqref{Edetached}, we conclude that~\eqref{pjunbounded} holds true. Note that from~\eqref{pjverticallysmall} and~\eqref{pjunbounded} it follows in particular that, up to an unrelabeled subsequence,
\begin{equation} \label{|pj'|diverges}
\lim_{j \rightarrow +\infty} |p_j| = \lim_{j \rightarrow +\infty} |p_j'| = +\infty.
\end{equation}

\item[] \textbf{Step 3. Ruling out ``graph-type'' surfaces.} Let now~$\{ \ell_j \}_{j \in \N} \subset (0, 1)$ be a sequence fulfilling the requirements
\begin{equation} \label{reqonelljandRj}
\lim_{j \rightarrow +\infty} \ell_j = 0 \quad \mbox{and} \quad \lim_{j \rightarrow +\infty} j \ell_j = \lim_{j \rightarrow +\infty} \ell_j^{n + 2} |p_j|^{1 - \alpha} = +\infty.
\end{equation}
Note that such a sequence exists, by virtue of~\eqref{|pj'|diverges}---indeed, one can take for instance
$$
\ell_j \coloneqq \max \Bigg\{ {\frac{1}{\sqrt{j + 1}}, \frac{1}{\big( {|p_j| + 2} \big)^{\frac{1 - \alpha}{2(n + 2)}}}} \Bigg\}.
$$
Define
\begin{equation} \label{LambdaCjdefs}
\Lambda_j \coloneqq \Big\{ {x \in \R^{n + 1} \mbox{ s.t. } x_{n + 1} < - \ell_j |x'|} \Big\} \quad \mbox{and} \quad \cC_j \coloneqq \Lambda_j \cap B_{4 |p_j|},
\end{equation}
for~$j \in \N$. We claim that
\begin{equation} \label{Cjnotincuded}
\mbox{there exists } j_1 \in \N \mbox{ such that } \cC_j + p_j \not\subset E_j \mbox{ for every integer } j \ge j_1.
\end{equation}
To prove this, we argue by contradiction and suppose that, up to an unrelabeled subsequence,
\begin{equation} \label{Ejisfat}
\cC_j + p_j \subset E_j \quad \mbox{for every large } j.
\end{equation}
Consider then the translated rescalings
\begin{equation} \label{Etildejdef}
\widetilde{E}_j \coloneqq \frac{E_j - p_j}{2 |p_j|}.
\end{equation}
Each~$\widetilde{E}_j$ has vanishing~$s$-mean curvature in the viscosity sense. Moreover, thanks to~\eqref{Ejisfat} we have that
$$
\Big\{ {x \in B_2 \mbox{ s.t. } x_{n + 1} < - \ell_j |x'|} \Big\} \subset \widetilde{E}_j,
$$
for all large~$j$'s. By this,~\eqref{EjinLj},~\eqref{pjverticallysmall}, and~\eqref{|pj'|diverges}, it is immediate to see that
$$
\bigg\{ {x \in B_2 \mbox{ s.t. } x_{n + 1} < - 2 \ell_j} \bigg\} \subset \widetilde{E}_j \cap B_2 \subset \bigg\{ {x \in B_2 \mbox{ s.t. } x_{n + 1} < \frac{1}{j}} \bigg\},
$$
if~$j$ is large enough.

By virtue of this and the fact that~$\{ \ell_j \}$ is infinitesimal by~\eqref{reqonelljandRj}, we can apply the~$\varepsilon$-regularity theory for flat hypersurfaces with vanishing~$s$-mean curvature and deduce that, for every~$j \in \N$ sufficiently large, there exists a function~$\varphi_j \in C^{1, \frac{s}{2}}(B_1')$ satisfying
\begin{equation} \label{varphijC1betabound}
\| \nabla \varphi_j \|_{C^{\frac{s}{2}}(B_1')} \le C_1
\end{equation}
and such that, up to negligible sets,
\begin{equation} \label{Ejtildesubgraph}
\widetilde{E}_j \cap \Big( {B'_1 \times \big( {- \sigma_1, \sigma_1} \big)} \Big) = \Big\{ {x \in B'_1 \times \big( {- \sigma_1, \sigma_1} \big) \mbox{ s.t. } x_{n + 1} < \varphi_j(x')} \Big\},
\end{equation}
for some constants~$C_1 \ge 1$ and~$\sigma_1 \in (0, 1)$ depending only on~$n$ and~$s$. This can be obtained directly from~\cite[Theorem~6.1]{CRS10} (which is stated for minimisers of the~$s$-perimeters but actually holds for mere viscosity solutions of the Euler-Lagrange equation) or from the more general and recent~\cite[Theorem~1.10]{M25}.

Going back to~$E$ (through transformations~\eqref{Ejdef} and~\eqref{Etildejdef}), this entails that~$E$ is a subgraph inside every cylinder
$$
B'_{2 |p_j|}(p_j') \times \bigg( {(p_j)_{n + 1} - \frac{1}{j} - 2 \sigma_1 |p_j|, (p_j)_{n + 1} - \frac{1}{j} + 2 \sigma_1 |p_j|} \bigg).
$$
Since, by~\eqref{|pj'|diverges}, these cylinders exhaust the whole of~$\R^{n + 1}$---for instance because the above cylinder contains~$B_{|p_j| / 2}' \times \big( - \sigma_1 |p_j|, \sigma |p_j| \big)$ for every~$j$ sufficiently large---, we conclude that~$E$ is globally the subgraph of a function~$\varphi: \R^n \to \R$ of class~$C^{1, \frac{s}{2}}$. By linking this piece of information to~\eqref{Ejtildesubgraph}, we see that
$$
\varphi(x') = (p_j)_{n + 1} - \frac{1}{j} + 2 |p_j| \varphi_j \bigg( {\frac{x' - p'}{2 |p_j|}} \bigg) \quad \mbox{for all } x' \in B'_{\frac{|p_j|}{2}},
$$
for~$j$ large enough. But then,~\eqref{varphijC1betabound} gives that
$$
[\nabla \varphi]_{C^{\frac{s}{2}} \big( {B'_{|p_j|/4}} \big)} \le \frac{C_1}{|p_j|^{\frac{s}{2}}} \quad \mbox{for every large } j.
$$
Recalling again~\eqref{|pj'|diverges}, we obtain that~$\nabla \varphi$ is constant in~$\R^n$ and thus that~$E$ is a half-space, contradicting our standing assumption~\eqref{Enothalfspace}. As a result, claim~\eqref{Cjnotincuded} must be true.

\item[] \textbf{Step 4. Turning~$F_{\varepsilon_j}$ into a supersolution: looking for some extra mass outside~$E_j$.}
We now claim that, for~$j$ large enough, there exists a set~$D_j$ such that
\begin{equation} \label{keyclaim}
\dist \big( {D_j, p_j} \big) > 0, \quad D_j \subset F_{\varepsilon_j} \setminus E_j, \quad \mbox{and} \quad \int_{D_j} \frac{\dd y}{|y - p_j|^{n + 1 + s}} \ge \frac{c_2}{|p_j|^{s + 1 - \alpha}},
\end{equation}
for some constant~$c_2 \in (0, 1)$ depending only on~$n$ and~$s$. To find such a set~$D_j$, we first observe that, defining
$$
\widehat{\Lambda}_j \coloneqq \bigg\{ {x \in \R^{n + 1} \mbox{ s.t. } x_{n + 1} < - \frac{\ell_j}{4} \, |x'|} \bigg\},
$$
it holds
\begin{equation} \label{ChatjinFj}
\widehat{\Lambda}_j + p_j \subset F_{\varepsilon_j} \quad \mbox{for every large } j.
\end{equation}%
To verify this, let~$x \in \widehat{\Lambda}_j + p_j$, i.e.,~$x \in \R^{n + 1}$ such that~$x_{n + 1} + \frac{\ell_j}{4} \, |x' - p_j'| < (p_j)_{n + 1}$. Of course, if~$x' \in \R^n \setminus B_{4 (p_j)_{n + 1} / \ell_j}(p_j')$, then~$x_{n + 1} < 0$ and therefore~$x \in \R^{n + 1}_- \subset F_{\varepsilon_j}$. Conversely, let~$x' \in B_{4 (p_j)_{n + 1} / \ell_j}(p_j')$. Notice that
$$
|x'| \ge |p_j'| - |x' - p_j'| \ge |p_j'| - \frac{4 (p_j)_{n + 1}}{\ell_j} \ge |p_j'| - \frac{4}{j \ell_j} \ge 1,
$$
if~$j$ is sufficiently large, independently of~$x$, thanks to~\eqref{pjverticallysmall},~\eqref{|pj'|diverges}, and~\eqref{reqonelljandRj}. By taking advantage of this, along with~\eqref{epsilonsmallerthan1overj}, the identity~$(p_j)_{n + 1} = \varepsilon_j^{1 - \alpha} |p_j'|^\alpha$ (which holds since~$p_j \in \partial F_{\varepsilon_j}$), and the elementary inequality~$(1 + t)^\alpha \le 1 + \alpha t$ for every~$t \ge 0$, we compute
\begin{align*}
x_{n + 1} + \frac{\ell_j}{4} \, |x' - p_j'| & < (p_j)_{n + 1} = \varepsilon_j^{1 - \alpha} |p_j'|^\alpha \le \varepsilon_j^{1 - \alpha} \big( {|x'| + |p'_j - x'|} \big)^\alpha \\
& = \varepsilon_j^{1 - \alpha} |x'|^\alpha \bigg( {1 + \frac{|x' - p'_j|}{|x'|}} \bigg)^{\! \alpha} \le \varepsilon_j^{1 - \alpha} |x'|^\alpha \bigg( {1 + \alpha \, \frac{|x' - p'_j|}{|x'|}} \bigg) \\
& \le \varepsilon_j^{1 - \alpha} \bigg( {|x'|^\alpha + \frac{|x' - p'_j|}{|x'|^{1 - \alpha}}} \bigg) \le \varepsilon_j^{1 - \alpha} |x'|^\alpha + \frac{1}{j} |x' - p'_j|.
\end{align*}
Recalling~\eqref{reqonelljandRj}, this gives
$$
x_{n + 1} < \varepsilon_j^{1 - \alpha} |x'|^\alpha - \frac{1}{j} \left( \frac{j \ell_j}{4} - 1 \right) |x' - p_j'| < \varepsilon_j^{1 - \alpha} |x'|^\alpha,
$$
that is~$x \in F_{\varepsilon_j}$, if~$j$ is large enough, independently of~$x$. Thus,~\eqref{ChatjinFj} is proved and we can now address the construction of a set~$D_j$ satisfying~\eqref{keyclaim}.

\item[] \textbf{Step 5a. Construction of the set~$D_j$: the density estimates case.}
In this step, we assume that~\ref{tech-a} is in force. In order to construct~$D_j$ in this case, we observe that property~\eqref{Cjnotincuded} leaves us with two possibilities: for each~$j \ge j_1$, either
\begin{enumerate}[label=$(\roman*)$,leftmargin=23pt]
\item \label{casei} $\cC_j + p_j \subset \R^{n + 1} \setminus \overline{E_j}$,$ \,$ or
\item \label{caseii} $(\cC_j + p_j) \cap \partial E_j \ne \varnothing$.
\end{enumerate}

When~\ref{casei} occurs, we simply take
$$
D_j \coloneqq B_{\frac{|p_j|}{8}} \Big( {p_j - \frac{|p_j|}{2} \, e_{n + 1}} \Big).
$$
A straightforward verification shows that~$D_j \subset \cC_j + p_j \subset \widehat{\Lambda}_j + p_j$---and hence~$D_j \subset F_{\varepsilon_j} \setminus E_j$, recalling~\eqref{ChatjinFj} and~\ref{casei}, if~$j$ is sufficiently large. Moreover,~$|D_j| = 8^{- n - 1} |B_1| |p_j|^{n + 1}$ and~$D_j \subset B_{|p_j|}(p_j)$, from which it follows
$$
\int_{D_j} \frac{\dd y}{|y - p_j|^{n + 1 + s}} \ge \frac{|D_j|}{|p_j|^{n + 1 + s}} = \frac{|B_1|}{8^{n + 1}} \frac{1}{|p_j|^s} \ge \frac{|B_1|}{8^{n + 1}} \frac{1}{|p_j|^{s + 1 - \alpha}},
$$
thanks to~\eqref{|pj'|diverges}. Thus,~\eqref{keyclaim} holds true when~\ref{casei} is in force.

The construction of~$D_j$ in case~\ref{caseii} is a bit more involved. Take a point~$q_j \in (\cC_j + p_j) \cap \partial E_j$ and define
$$
D_j \coloneqq B_{\frac{(p_j - q_j)_{n + 1}}{2}}(q_j) \setminus E_j.
$$
By definition,~$D_j \subset \R^{n + 1} \setminus E_j$. In addition, using that~$q_j \in \cC_j + p_j$, for any~$x \in D_j$ we have
\begin{equation} \label{DjinB}
\begin{aligned}
|x - p_j| & \le |x - q_j| + |q_j - p_j| < \frac{1}{2} \, \big( {p_j - q_j} \big)_{n + 1} + \sqrt{|q_j' - p_j'|^2 + \big( {q_j - p_j} \big)_{n + 1}^2} \\
& < \Bigg( {\frac{1}{2} + \sqrt{\frac{1}{\ell_j^2} + 1}} \, \Bigg) \big( {p_j - q_j} \big)_{n + 1} \le \frac{4}{\ell_j} \, \big( {p_j - q_j} \big)_{n + 1}
\end{aligned}
\end{equation}
and
\begin{equation} \label{DjinLambdahat}
\begin{aligned}
& \big( {x - p_j} \big)_{n + 1} + \frac{\ell_j}{4} \, |x' - p_j'|\\
& \hspace{40pt} \le \big( {x - q_j} \big)_{n + 1} + \big( {q_j - p_j} \big)_{n + 1} + \frac{\ell_j}{4} \, |x' - q_j'| + \frac{\ell_j}{4} \, |q_j' - p_j'| \\
& \hspace{40pt} \le \frac{5}{4} \, |x - q_j| + \big( {q_j - p_j} \big)_{n + 1} + \frac{1}{4} \, \big( {p_j - q_j} \big)_{n + 1} < - \frac{1}{8} \, \big( {p_j - q_j} \big)_{n + 1} < 0.
\end{aligned}
\end{equation}
This latter estimate yields that~$D_j \subset \widehat{\Lambda}_j + p_j$, from which, recalling~\eqref{ChatjinFj}, it follows that~$D_j \subset F_{\varepsilon_j}$. In order to establish the lower bound for the integral in~\eqref{keyclaim}, we first notice that, in view of assumption~\ref{tech-a}, the complement~$\R^{n + 1} \setminus E_j$ enjoys uniform density estimates around~$q_j$ at all scales. Hence,~$|D_j| \ge c_3 \big( {(p_j - q_j)_{n + 1}} \big)^{n + 1}$ for some constant~$c_3 \in (0, 1)$ independent of~$j$. On the other hand,~\eqref{DjinB} gives that~$D_j \subset B_{\frac{4}{\ell_j} (p_j - q_j)_{n + 1}}(p_j)$. Consequently,
\begin{align*}
\int_{D_j} \frac{\dd y}{|y - p_j|^{n + 1 + s}} & \ge \frac{|D_j|}{\Big( {\frac{4}{\ell_j} \big( {p_j - q_j} \big)_{n + 1}} \Big)^{n + 1 + s}} \ge \frac{c_3}{4^{n + 2}} \frac{\ell_j^{n + 2}}{\big( {p_j - q_j} \big)_{n + 1}^s} \\
& \ge \frac{c_3}{4^{n + 3}} \frac{\ell_j^{n + 2}}{|p_j|^s} \ge \frac{c_3}{4^{n + 3}} \frac{1}{|p_j|^{s + 1 - \alpha}},
\end{align*}
if we take~$j$ large enough. Note that the second-to-last inequality follows using that~$q_j \in B_{4 |p_j|}(p_j)$ (recall that~$q_j \in \cC_j + p_j$ and definition~\eqref{LambdaCjdefs} of~$\cC_j$), while for the last inequality we took advantage of condition~\eqref{reqonelljandRj} on~$\ell_j$. Therefore, claim~\eqref{keyclaim} holds also when~\ref{caseii} is assumed.

\item[] \textbf{Step 5b. Construction of the set~$D_j$: the~$C^1$~case.} To obtain the set~$D_j$ when~\ref{tech-b} is assumed, we use a variant of the sliding method. Consider the \emph{ice cream cone} shaped sets
$$
\mathscr{G}_{r, j} \coloneqq \bigcup_{x \in \left( \Lambda_j \cap B_r \right) + p_j} B_{\frac{(p_j - x)_{n + 1}}{4}}(x) \quad \mbox{for } r > 0.
$$
Note that the sets~$\{ \mathscr{G}_{r, j} \}_{r > 0}$ are expanding as~$r$ grows. Also, simple computations show that
\begin{equation} \label{GinLambdahat}
\big({\Lambda_j \cap B_r }\big) + p_j \subset \mathscr{G}_{r, j} \subset \big({\widehat{\Lambda}_j \cap B_{2 r}}\big) + p_j \quad \mbox{for all } r > 0.
\end{equation}

Let
$$
\mathcal{R}_j \coloneqq \Big\{ {r > 0 \mbox{ s.t. } \mathscr{G}_{r, j} \subset E_j} \Big\}.
$$
We claim that the set~$\mathcal{R}_j$ is a non-empty interval for every large~$j$. To see this, we observe that, as a consequence of the~$C^1$-regularity of~$\partial E$ prescribed by hypothesis~\ref{tech-b} and the fact that~$\partial E_j$ is touched from above at~$p_j$ by the graph of a smooth function, for~$j$ large enough it holds
$$
E_j \cap B_{\delta_j}(p_j) = \Big\{ {x \in B_{\delta_j}(p_j) \mbox{ s.t. } x_{n + 1} < u_j(x')} \Big\},
$$
for some small~$\delta_j \in (0, 1) $ and some function~$u_j \in C^1(\R^n)$. As~$p_j \in \partial E_j \cap \partial F_{\varepsilon_j}$ and~$\overline{E_j} \subset \overline{F_{\varepsilon_j}}$, we have that~$u_j(p_j') = \varepsilon_j^{1 - \alpha} |p_j'|^\alpha$,~$\nabla u_j(p_j') = \alpha \varepsilon_j^{1 - \alpha} |p_j'|^{\alpha - 2} p_j'$, and~$0 \le \varepsilon_j^{1 - \alpha} |x'|^\alpha - u_j(x') = o \big( {|x' - p_j'|} \big)$ as~$x' \rightarrow p_j'$. Thanks to these facts and arguing almost exactly as in the proof of~\eqref{ChatjinFj} given earlier, one sees that
$$
\big( {\widehat{\Lambda}_j \cap B_{2 r}} \big) + p_j \subset E_j,
$$
for every~$r \in \big( {0, \frac{\delta_j}{2}} \big]$ sufficiently small in dependence of~$j$, provided~$j$ is large enough. From this and the right-most inclusion in~\eqref{GinLambdahat} it follows that~$\mathcal{R}_j$ is non-empty, provided~$j$ is taken sufficiently large.

Let then
$$
r_j \coloneqq \sup \mathcal{R}_j.
$$
As~$\mathcal{R}_j \ne \varnothing$, we know that~$r_j$ is a well-defined positive real number if~$j$ is large. Also,~$r_j$ is finite and smaller or equal to~$4 |p_j|$, in view of the left-most inclusion in~\eqref{GinLambdahat}, claim~\eqref{Cjnotincuded}, and definition~\eqref{LambdaCjdefs}. Hence,
\begin{equation} \label{rjle2Rjpj}
r_j \in \big( {0, 4 |p_j|} \big] \quad \mbox{for every~$j \in \N$ sufficiently large}.
\end{equation}
In consequence of this, it is not hard to see that there exist a point~$q_j \in \partial \, \mathscr{G}_{r_j, j} \cap \partial E_j$ and, corresponding to it, another point~$x_j \in \big( {\overline{\Lambda_j} \cap \partial B_{r_j}} \big) + p_j$ such that~$B_{\frac{(p_j - x_j)_{n + 1}}{4}}(x_j) \subset E_j$ and~$q_j \in \partial B_{\frac{(p_j - x_j)_{n + 1}}{4}}(x_j) \cap \partial E_j$.

We define
$$
D_j \coloneqq B_{\frac{(p_j - x_j)_{n + 1}}{4}}(q_j) \setminus E_j.
$$
Using the information on the positions of~$q_j$ and~$x_j$ obtained in the previous paragraph, computations similar to~\eqref{DjinB}-\eqref{DjinLambdahat} give that~$D_j \subset \big( { \widehat{\Lambda}_j \cap B_{\frac{4}{\ell_j} (p_j - x_j)_{n + 1}} } \big) + p_j$. In particular, we have that~$D_j \subset F_{\varepsilon_j}$, recalling~\eqref{ChatjinFj}. Furthermore, as~$E_j$ is touched from the inside at~$q_j \in \partial E_j$ by the ball~$B_{\frac{(p_j - x_j)_{n + 1}}{4}}(x_j)$, Proposition~\ref{prop:densestfors-stat} yields that~$|D_j| \ge c_4 \big( {(p_j - x_j)_{n + 1}} \big)^{n + 1}$ for some constant~$c_4 \in (0, 1)$ depending only on~$n$ and~$s$. By these facts, along with the inclusion~$x_j \in \partial B_{r_j}(p_j)$ and the bound~\eqref{rjle2Rjpj}, we obtain that
\begin{align*}
\int_{D_j} \frac{\dd y}{|y - p_j|^{n + 1 + s}} & \ge \frac{|D_j|}{\Big( {\frac{4}{\ell_j} \big( {p_j - x_j} \big)_{n + 1}} \Big)^{n + 1 + s}} \ge \frac{c_3}{4^{n + 2}} \frac{\ell_j^{n + 2}}{\big( {p_j - x_j} \big)_{n + 1}^s} \\
& \ge \frac{c_3}{4^{n + 2}} \frac{\ell_j^{n + 2}}{r_j^s} \ge \frac{c_3}{4^{n + 3}} \frac{\ell_j^{n + 2}}{|p_j|^s} \ge \frac{c_3}{4^{n + 3}} \frac{1}{|p_j|^{s + 1 - \alpha}},
\end{align*}
thanks to condition~\eqref{reqonelljandRj} and provided~$j$ is large enough. This completes the proof of claim~\eqref{keyclaim}.

\item[] \textbf{Step 6. Conclusion.} We are now in position to finish the proof. For~$j$ large, define
$$
\widetilde{F}_j \coloneqq F_{\varepsilon_j} \setminus D_j.
$$
By taking advantage of~\eqref{keyclaim} and estimate~\eqref{HsFepssmall} (which can be applied in view of~\eqref{|pj'|diverges}), we see that
$$
\mathrm H_{s, \widetilde{F}_j}(p_j) = \mathrm H_{s, F_{\varepsilon_j}}(p_j) + 2 \int_{D_j} \frac{\dd y}{|y - p_j|^{n + 1 + s}} \ge \frac{c_2 - C_0 \, \varepsilon_j^{1 - \alpha}}{|p_j|^{s + 1 - \alpha}} > 0,
$$
if~$j$ is large enough, thanks to~\eqref{epsjgoestozero}. But this is impossible, as~$E_j$ satisfies~$H_{s, E_j} \le 0$ at~$p_j$ in the viscosity sense and~$E_j \subset \widetilde{F}_j$, by~\eqref{keyclaim}. We thus found a contradiction and must conclude that claim~\eqref{fundamentalclaim} holds true. Consequently,~$\varepsilon_{j_0} = 0$ for some~$j_0 \in \N$, which, going back to definition~\eqref{epsjdef}, allows us to conclude that~$E_{j_0} \subset F_\varepsilon$ for every~$\varepsilon > 0$. Therefore,~$E_{j_0} \subset \R^{n + 1}_-$ or, equivalently,~$E \subset \R^n \times \left( - \infty, - \frac{1}{j_0} \right)$. Since this contradicts~\eqref{Eclosetohorhyp}, assumption~\eqref{Enothalfspace} must necessarily be incorrect. Hence,~$E$ is a half-space and the proof of the theorem is complete.\qedhere
\end{itemize}
\end{proof}

\end{document}